\documentclass[english,12pt]{amsart}

\usepackage{amssymb,mathtools,amsmath}

\usepackage[utf8]{inputenc}
\usepackage[T1]{fontenc}
\usepackage[english]{babel}

\usepackage[inline]{enumitem}
\setenumerate[1]{font=\color{black}, label=\textbf{\arabic*.}}

\usepackage{yhmath}
\usepackage{enumitem}
\setlist{topsep=-0pt} 

\makeatletter
\def\namedlabel#1#2{\begingroup
    #2%
    \def\@currentlabel{#2}%
    \phantomsection\label{#1}\endgroup
}
\makeatother
\usepackage{graphicx}
\usepackage{subcaption}

\usepackage{euscript}

\usepackage{fouriernc}
\usepackage[scaled=0.875]{helvet}
\usepackage{courier}

\usepackage{ marvosym }

\usepackage[babel=true]{csquotes}

\usepackage[marginparwidth=2cm]{geometry}
\usepackage[usenames,dvipsnames,svgnames,table]{xcolor}

\usepackage[normalem]{ulem}
\usepackage[marginpar]{todo}

\usepackage[colorlinks=true,urlcolor=DarkBlue,linkcolor=DarkRed,citecolor=DarkGreen,pagebackref=true]{hyperref}

\newcommand{\Rd}{\color{RedOrange}}

\newcommand{\Bk}{\color{black}}

\usepackage{tikz}
\usepackage{tikz-cd}
\usepackage{fp}
\usetikzlibrary{calc}
\usetikzlibrary{shapes,arrows}
\usetikzlibrary{fit,positioning}
\usetikzlibrary{patterns,decorations.pathreplacing}
\usetikzlibrary{arrows.meta}
\makeatletter
\def\calcLength(#1,#2)#3{%
\pgfpointdiff{\pgfpointanchor{#1}{center}}%
             {\pgfpointanchor{#2}{center}}%
\pgf@xa=\pgf@x%
\pgf@ya=\pgf@y%
\FPeval\@temp@a{\pgfmath@tonumber{\pgf@xa}}%
\FPeval\@temp@b{\pgfmath@tonumber{\pgf@ya}}%
\FPeval\@temp@sum{(\@temp@a*\@temp@a+\@temp@b*\@temp@b)}%
\FProot{\FPMathLen}{\@temp@sum}{2}%
\FPround\FPMathLen\FPMathLen5\relax
\global\expandafter\edef\csname #3\endcsname{\FPMathLen}
}
\makeatother

\usepackage{xkeyval}
\usepackage{moreverb}
\usepackage{epic}
\usepgfmodule{shapes,plot,decorations}
\usepackage{accents}

\usepackage{epsfig}
\usepackage{pinlabel}

\usepackage{nicefrac}

\newcommand{\Z}{\mathbb{Z}}

\newcommand{\R}{\mathbb{R}}

\newcommand{\Hb}{\mathbb{H}}

\newcommand{\eps}{\varepsilon}

\newcommand{\eucl}{\mathrm{eucl}}

\newcommand{\SL}{\mathrm{SL}}
\newcommand{\PSL}{\mathrm{PSL}}
\newcommand{\PSO}{\mathrm{PSO}}
\newcommand{\SO}{\mathrm{SO}}

\newcommand*{\ps}[1]{\mathrm{PSL}_{#1}(\mathbb{R})}

\newcommand{\LG}{\Lambda_{\Gamma}}
\newcommand{\G}{\Gamma}

\renewcommand{\O}{\Omega}

\newcommand{\dO}{\partial \Omega}

\newcommand{\dmax}{d_{\mathcal{O}_{\mathrm{max}}}}

\newcommand{\PP}{\mathbb{P}}

\newcommand{\Aut}{\textrm{Aut}}

\newcommand{\Vol}{\textrm{Vol}}
\newcommand{\Span}{\textrm{Span}}

\newcommand{\Stab}{\textrm{Stab}}

\newcommand{\Cc}{\mathcal{C}}

\definecolor{light-gray}{gray}{0.80}
\definecolor{dark-gray}{gray}{0.30}

\renewcommand{\leq}{\leqslant}
\renewcommand{\geq}{\geqslant}

\newcommand{\PGL}{\mathrm{PGL}}
\newcommand{\GL}{\mathrm{GL}}

\usepackage{aurical}

\newcommand{\mr}{\mathrm}
\newcommand{\mc}{\mathcal}
\newcommand{\mb}{\mathbb}
\newcommand{\dni}{\partial_{\mathrm{ni}}}
\newenvironment{sm}
    {\left(
    \begin{smallmatrix}
    }
    {
    \end{smallmatrix} 
    \right)}

\theoremstyle{plain}
\newtheorem{theorem}{Theorem}[section]

\newtheorem{proposition}[theorem]{Proposition}
\newtheorem{corollary}[theorem]{Corollary}
\newtheorem{lemma}[theorem]{Lemma}
\newtheorem{fact}[theorem]{Fact}

\theoremstyle{definition}

\theoremstyle{remark}
\newtheorem{remark}[theorem]{Remark}

\subjclass[2010]{20F55, 20F65, 20H10, 22E40, 51F15, 53C50, 57M50, 57S30}

\title[]{On 4-dimensional convex projective domains invariant by a lattice of $\mathrm{SL}_2 (\mathbb{R})$}

\author{
Pierre-Louis Blayac
}
\address{Univ Strasbourg, CNRS, IRMA - UMR 7501, Strasbourg, France}
\email{blayac@unistra.fr}

\author{
Ludovic Marquis
}
\address{Univ Rennes, CNRS, IRMAR - UMR 6625, F-35000 Rennes, France}
\email{ludovic.marquis@univ-rennes1.fr}

\definecolor{mylightgray}{RGB}{60,179,113}
\definecolor{mygray}{RGB}{119,136,153}

\definecolor{myconnection}{RGB}{47,79,79}

\begin{document}
\maketitle

\today

\begin{abstract}
This paper is a sequel to the Erratum \cite{Erratum} to the paper \cite{CM2014finitude}.
The main result of the Erratum was relating several notions of geometrical finiteness in round convex projective geometry and we prove here that our series of implications was sharp, by providing counterexamples to the implications that were not established.
Our counterexamples are 4-dimensional convex domains $\O$ acted on by $\rho (\G)$ where $\G$ is a lattice of $\SL_2 (\R)$ and $\rho$ is the irreducible representation of $\SL_2 (\R)$ of dimension $5$.
We give a description of all $\rho(\G)$-invariant convex domains, and in particular we construct one which is "close enough" to the convex hull $\mc C$ of the limit set of $\rho(\G)$ so that the Hilbert volume $\Vol_{\O/\G}(\mc C/\G)$ of the convex core is infinite.
We include an appendix with a smoothing procedure in the spirit of \cite{CLT18,DGK17}.
\end{abstract}

\section{Introduction}

Let $\O$ be an open subset of $\R\PP^d$ which is properly convex, i.e.~contained in an affine chart of $\R\PP^d$ where it is convex and bounded.
We say $\O$ is \emph{round} if its boundary is differentiable and strictly convex (any segment contained in the boundary must be reduced to a point).
Finally, let $\Gamma$ be a discrete subgroup of $\Aut(\O)\subset\PGL_{d+1}(\R)$, the group of projective transformations that preserve $\Omega$.
The domain $\O$ carries a natural $\Aut(\O)$-invariant proper metric called the Hilbert metric, defined in terms of projective cross-ratios.
This metric is very useful when studying $\G$ and $\O$; a first basic use is for instance to prove that $\G$ acts properly discontinuously on $\O$, making $\O/\G$ an orbifold called a convex projective orbifold.
For more detailed reminders on convex projective geometry and the Hilbert metric, see \cite[\S2]{CM2014finitude}.

Convex projective manifolds form an active field of research.
Its rich and beautiful pool of examples, see e.g.\ \cite{benoist_survey,HandbookHilbert}, makes it a good testing ground before the larger field of discrete subgroups of Lie groups.
The first basic example is the ellipsoid, which is isometric to the real hyperbolic space $\mb H^d$.
In the present paper we are interested in the following more subtle example: suppose $d$ is even and identify $\R^{d+1}$ with the space of degree $d$ homogeneous polynomials in two variables.
Let $\rho:\GL_2(\R)\to\GL_{d+1}(\R)$ be the irreducible representation acting as 
\begin{equation}\label{irreducible rep}
\rho\begin{sm}
a&b\\c&d
\end{sm}\cdot P(X,Y)
=P\left((X,Y)\cdot \begin{sm}
a&b\\c&d
\end{sm}\right)
=P(aX+cY,bX+dY).
\end{equation}
Then the image of $\rho$ preserves the cone of polynomials that take nonnegative values.
Obviously this cone is convex, closed and does not contain any line, so its projectivisation is properly convex and we denote by $\mc O_{\max}$ its interior (which is nonempty).
If $d=2$ then $\mc O_{\max}$ is a disc and we retrieve the real hyperbolic plane.
In higher dimension, note that the point $[X^d]\in\partial\mc O_{\max}$ is stabilised by the image under $\rho$ of triangular matrices, inducing a $\rho$-equivariant embedding from the circle $\GL_2(\R)/\{\begin{sm}*&*\\0&*\end{sm}\}$ into $\partial\mc O_{\max}$, whose image is denoted by $\Lambda$, and hence $\rho(\GL_2(\R))$ preserves $\Lambda$ as well as its convex hull, whose interior $\mc O_{\min}\subset\mc O_{\max}$ is another invariant properly convex domain.
In fact $\rho(\GL_2(\R))$ preserves many intermediate properly convex domains: the uniform $R$-neighborhoods $\mc O_R$, with $R>0$, of $\mc O_{\min}$ in $\mc O_{\max}$ for the Hilbert metric, which turn out to be round.

We are here interested in the case where $d=4$ and we study the action of $\rho(\G)$ where $\G$ is a noncocompact lattice of $\SL_2(\R)$.
The reason is that these give rise to the simplest example of round convex projective cusps that are not real hyperbolic cusps.
Indeed, inspired from real hyperbolic and negatively curved Riemannian geometry \cite{BowditchGF,Bowditch_GF}, there exist several notions geometrically finite \emph{round} convex projective orbifolds,
developped and studied by Crampon--Marquis \cite{CM2014finitude} and also studied by Cooper--Long--Tillman \cite{CLT2015cvxisom} in parallel.

Among the many characterisations of geometric finiteness in Riemannian geometry, one stands out, due to Beardon and Maskit (Property F2 in \cite{Bowditch_GF}), because it only uses the action of the group on the boundary at infinity, and hence it extends to the very general setting of convergence group actions.
Crampon--Marquis observe that this general setting also encompasses $\G$ acting on $\partial\O$ for $\O$ round and $\G\subset\Aut(\O)$ discrete and derived a notion of geometric finiteness that we will call here \emph{weakly geometrically finite} and denote \eqref{item:GF_on_boundary}.
They also introduced a stronger notion by adding a hypothesis, and we will call this notion \emph{strongly geometrically finite} and denote it \eqref{item:GF_on_Omega}.
Crampon--Marquis proved that the main object of study of the present paper $\mc O_R/\rho(\G)$, where $\G$ is a noncocompact lattice of $\SL_2(\R)$, is weakly but not strongly geometrically finite.

Bowditch's other characterisations of geometric finiteness were also adapted to round convex projective geometry and were shown to be equivalent to strong geometric finiteness.
However there was a mistake in the proof and the statement is not correct, as explained in \cite{Erratum} where a new, more complicated, diagram of implications between the various notions of geometrical finiteness is established, see Figure~\ref{fig:new pattern} (we provide extra explanations at the end of the introduction).
In this diagram, the implications that are not equivalences have been colored in red or green.
The converses of the two red implications were already disproved in \cite{Erratum}.
The goal of the present paper is to give counterexamples to converses of the four green implications.

\begin{figure}
\begin{tikzpicture}[>=latex,scale=1.1]
\tikzstyle{arrondi}=[rounded corners=4pt]
\tikzstyle{implique}=[-{Implies},double equal sign distance]
\tikzstyle{equivalent}=[implies-implies,double equal sign distance]
	\node (TF) at (2,1) {\eqref{item:region_parabolique}};
	\node (GF) at (2,-1) {\eqref{item:GF_on_Omega}};
	\node (HC) at (0,0) {\eqref{item:gf_cusps_lobat}};
	\node (GF+Gen) at (-2,0) {\eqref{item:GF_on_Omega}\&\eqref{item:generic}};
	
	\node (VF) at (4,-1) {\eqref{item:convex_core}${}_1$};
	\node (gf+hyp) at (4,1) {\eqref{item:GF_on_boundary}$\&$\eqref{item:convex_core_ghyp}};
	
	\node (PEC) at (6,-1) {\eqref{item:thick_part}};
	\node (PNC) at (8,-1){\eqref{item:non_cusp_part}};
	\node (gf) at (6,1) {\eqref{item:GF_on_boundary}};
	\node (CU) at (8,1) {\eqref{item:cusp_uniform}};
	
	\node (tf) at (8.6,0) {\eqref{item:region_parabolique_gen}};

  \draw[implique] (HC) -- (GF);
  \draw[equivalent] (TF) -- (GF);
  \draw [implique,green!60!black] (GF) -- (VF) node[midway,below,sloped]{\tiny green};
  \draw[implique,green!60!black] (GF) -- (gf+hyp) node[midway,below,sloped]{\tiny green};
  \draw[implique] (gf+hyp) -- (gf);
  \draw [implique] (VF) -- (PEC);
  \draw [implique,green!60!black] (GF+Gen) -- (HC) node[midway,below,sloped]{\tiny green};
    \draw [equivalent,green!60!black] (gf) -- (CU) node[midway,below,sloped]{\tiny green};
    \draw [equivalent,green!60!black]  (PNC) -- (gf) node[midway,below,sloped]{\tiny green};
  \draw [equivalent,green!60!black] (gf) to[out=335,in=180] node[midway,below,sloped]{\tiny green} (tf);
    \draw [equivalent] (PNC) -- (PEC);

\draw[arrondi] (1.5,-1.3) rectangle (2.5,1.3);
\draw[arrondi] (5.5,-1.3) rectangle (9.12,1.3);

\end{tikzpicture}
\caption{New pattern of implications: black arrows where correctly proved in the former paper, the green ones repair the mistakes of the former paper.}
\label{fig:new pattern}
\end{figure}
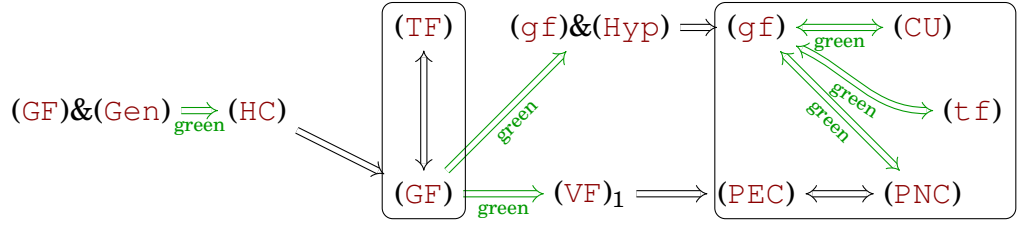

\begin{theorem}\label{thm:non_implication}
	Let $\O$ be a round convex set of $\R\PP^d$ and $\G \leqslant \Aut(\Omega)$. Then:
	\begin{enumerate}
		\item The condition  \eqref{item:convex_core}${}_1$ does not imply the condition  \eqref{item:GF_on_Omega}.
		\item The condition  \eqref{item:GF_on_boundary}$\&$\eqref{item:convex_core_ghyp} does not imply the condition  \eqref{item:GF_on_Omega}.
		\item The condition \eqref{item:GF_on_boundary} does not imply the condition \eqref{item:convex_core}${}_1$.
		\item The condition  \eqref{item:GF_on_boundary} does not imply the condition \eqref{item:GF_on_boundary}$\&$\eqref{item:convex_core_ghyp}.
	\end{enumerate}
	More precisely, consider the irreducible representation $\rho: \SL_2 (\R) \to \SL_5 (\R)$ with limit set $\Lambda$ (the image of the equivariant map $\SL_2(\R)/\{\begin{sm}*&*\\0&*\end{sm}\}\to\R\PP^4$), and let $\mc C$ be the convex hull of $\Lambda$ minus $\Lambda$ itself.
	Now let $\G$ be a noncocompact lattice of $\SL_2 (\R)$.
	Then the limit set of $\G$ is $\Lambda$ and we have the following.
	\begin{itemize}
	    \item[(a)] (\cite[Prop.\,10.6]{CM2014finitude}) For all $\rho(\G)$-invariant a round convex domain $\O\subset\R\PP^4$, the action of $\G$ on $\O$ is weakly but not strongly geometrically finite.
	
	    \item[(b)] For any $\rho(\SL_2(\R))$-invariant round convex domain $\mc O\subset\R\PP^4$, $\mc C$ is Gromov-hyperbolic for the Hilbert metric of $\mc O$ and the $1$-neighborhood of the convex core $\mc C/\rho(\G)$ of $\mc O/\rho(\G)$ is of finite volume for the Hilbert metric of $\mc O$.
	    
	    \item[(c)] There exist a $\rho(\G)$-invariant round convex domain $\O\subset\R\PP^4$ such that $\Cc$ is not Gromov-hyperbolic for the Hilbert metric of $\O$ and the convex core $\mc C/\rho(\G)$ of $\O/\rho(\G)$ is of infinite volume for the Hilbert metric of $\O$.
	\end{itemize}
\end{theorem}

Observe that two implications are in fact disproved by a single round convex projective manifold, and the two other are also disproved by a single (obviously different) manifold.
Moreover, the underlying groups are the same, and the convex cores are the same, but of course the Hilbert metrics and volumes induced on that convex core are different.

In fact, in general given a round convex projective manifold $\O/\G$ of dimension $d$, the limit set of $\G$ does not depend on $\O$, as it can be described as the closure of the set of attracting fixed points in $\R\PP^d$ of proximal elements of $\G$ (an element is proximal if it has an attracting fixed point).
Weak and strong geometric finiteness can be purely described in terms of the action of $\G$ on the convex hull of the limit set, so they also do not depend on $\O$.
The Hilbert metric induced by $\O$ on the convex core, on the other hand, does depend on $\O$, more precisely the smaller $\O$ is, the bigger the induced metric will be.
Similarly, the volume induced by the Hilbert metric is the Hausdorff measure of dimension $d$, and it also gets bigger when $\O$ gets smaller.

So to prove (c), the idea is to construct a $\G$-invariant domain $\O$ which is close enough to $\mc C$ so that $\mc C$ is not Gromov-hyperbolic and $\mc C/\G$ has infinite volume.
Setting $\dni\mc C=\partial\mc C\smallsetminus\Lambda$ (ni stands for "nonideal'' following terminology from \cite{DGK17}),
we will see that $\PSL_2(\R)$ acts transitively on $\dni\mc C$ with stabiliser of order 2, meaning $\dni\mc C$ identifies with a circle bundle over $\mb H^2$, and $\mc O_{\max}\smallsetminus\mc O_{\min}$ can be naturally parametrised by $\dni\mc C\times [0,1/2)$ where the second coordinate measures how far we are from $\dni\mc C$.
This allows us to describe the nonideal boundary $\partial\O\smallsetminus\Lambda$ via a height function $u:\dni\mc C/\G\to(0,1/2)$.
The smaller the height is, the closer we are to the convex core.
As $\G\subset\PSL_2(R)$ has finite covolume, $\mb H^2/\G$ decomposes into a compact part and finitely many cusps.
This induces a decomposition of $\dni\mc C/\G$ into a compact part and finitely many cuspidal parts.
Then on the compact part $u$ is bounded below by continuity, so we can only make $u$ small deep in the cusps.
We will describe, up to a multiplicative error term, how small $u$ can be, and then we will check that it can be small enough for the convex core to have infinite volume.

To finish this introduction we give a quick reminder on the notions involved in Figure~\ref{fig:new pattern} and Theorem~\ref{thm:non_implication}.
Let $\O/\G$ be a round convex projective orbifold.
\begin{itemize}
\item[(\namedlabel{item:GF_on_boundary}{\texttt{gf}})] : "weakly geometrically finite", every $x\in\Lambda_\Gamma$ is conical\footnote{the projection in $\O/\G$ of any ray in $\O$ converging to $x$ passes infinitely often in a compact set} or bounded parabolic\footnote{the stabiliser $\Stab_\G(x)$ acts properly discontinuously cocompactly on $\LG\smallsetminus\{x\}$}.

\item[(\namedlabel{item:GF_on_Omega}{\texttt{GF}})] : "strongly geometrically finite", every point of $\Lambda_\Gamma$ is conical or uniformly bounded parabolic\footnote{the stabiliser $\Stab_\G(x)$ acts properly discontinuously cocompactly on the stereographic projection from $x$ of the convex hull $\mc C(\LG)$ into $\partial\Omega\smallsetminus\{x\}$}.

\item[(\namedlabel{item:gf_cusps_lobat}{\texttt{HC}})] : "hyperbolic cusps",  \eqref{item:GF_on_boundary} holds and for each parabolic point $p$, the stabiliser $\Stab_\G(x)$ is conjugate into $\mathrm{O}_{d,1} (\R)$.

\item[(\namedlabel{item:region_parabolique}{\texttt{TF}})] :  "type fini", will not be used here.

\item[(\namedlabel{item:region_parabolique_gen}{\texttt{tf}})] : "weakly finite type", will not be used here.

\item[(\namedlabel{item:thick_part}{\texttt{PEC}})] : "partie épaisse compacte", the thick part of the convex core is compact, will not be used here.

\item[(\namedlabel{item:non_cusp_part}{\texttt{PNC}})] : "partie non cuspidale compacte", the non-cuspidal part of the convex core is compact, will not be used here.

\item[(\namedlabel{item:cusp_uniform}{\texttt{CU}})] : "cusp uniform", there exists a $\Gamma$-equivariant family of disjoint horoballs $H_p\subset\O$\footnote{the only thing we need is that it is convex and $\partial H_p\cap \partial\O=\{p\}$, so any ray entering $H_p$ either exits it for good after some time or stays forever and converges to $p$} centered at parabolic points $p$, such that the action of $\G$ on $\Cc(\LG)\smallsetminus \bigcup_p H_p$ is cocompact.

\item[(\namedlabel{item:convex_core}{\texttt{VF}})${}_R$] : "volume fini", $\G$ is finitely generated and the uniform $R$-neighborhood (for the Hilbert metric) of the convex core $\Cc(\LG)/\G$  is of finite volume.\footnote{Here our volume form is the Hausdorff measure of the Hilbert metric, see \cite[\S2.1]{CM2014finitude}.}

\item[(\namedlabel{item:convex_core0}{\texttt{VF}})${}_0$] : "volume fini",  $\G$ is finitely generated and the convex core $\Cc(\LG)/\G$ is of finite volume for the Hilbert volume form from $\O\cap {\rm Span}(\Cc(\LG))$.

\item[(\namedlabel{item:convex_core_ghyp}{\texttt{Hyp}})] : "hyperbolic convex core", the convex core $\Cc(\LG)$ is Gromov-hyperbolic\footnote{Recall that a geodesic metric space is Gromov-hyperbolic if for some $\delta$ all geodesic triangles are $\delta$-thin: any side is in the $\delta$-neighborhood of the union of the two other sides, see e.g.\ \cite[\S III.H.1]{bridsonhaeflige}.} for the Hilbert metric of $\O$.

\item[(\namedlabel{item:generic}{\texttt{Gen}})] : "genericity", the limit set $\LG$ spans $\R\PP^d$, or its dual spans the dual projective space, will not be used here.\footnote{Recall that $\G$ also preserves a round properly convex open set $\O^*$ in the projective space of linear forms on $\R^{d+1}$, and hence has a limit set there too, see \cite[\S2.3]{CM2014finitude}.}
\end{itemize}

To conclude this introduction, let us mention a few recent results concerning geometric finiteness in convex projective geometry.
In \cite{CooperHeisenberg}, Cooper found an example of strictly convex projective cusp whose holonomy is nilpotent but not virtually abelian, and in \cite[Th.\,6.6]{Flechelles}, Fléchelles explains how to deform this construct to obtain a round domain.
In fact Fléchelles announced that by generalizing Cooper's method, he was
able to build representations of any finitely generated virtually
nilpotent group that are the holonomy of round elementary cusps.
Fléchelles---Islam---Zhu also announced that given any round elementary
cusp, they can produce geometrically finite nonelementary round convex
projective orbifolds that admit a cusp whose holonomy is that of the
initial round elementary cusp.

There are also notions of nonround convex projective cusps, but the situation is more complicated.
First there is the case of strictly convex domains that are not necessarily round.
Cooper--Long--Tillman studied them in \cite{CLT2015cvxisom} and showed that in the geometrically finite case with cusps of maximal rank, the domain must be round.
On the other hand, Fléchelles and Fléchelles---Islam---Zhu found (many) strictly convex examples with nonmaximal rank cusps that cannot be made round.

Then there is the nonstrictly convex case.
In \cite{CLT18} Cooper--Long--Tillman introduced generalized cusps of maximal ranks that are not necessarily strictly convex.
They showed these are always virtually abelian and then Ballas--Cooper--Leitner classified them in \cite{BCL20}.
Flechelles found many examples of generalized nonstrictly convex cusps of nonmaximal rank, where "many'' means with a great diversity of cusp groups: e.g. he can realise any group isomorphic to the holonomy of a round cusp, and he has an example where the cusp group is solvable but not virtually nilpotent.

The second author thanks A. Zimmer for pointing out to him the second point of Theorem~\ref{thm:non_implication} using the same $\O_0$ that we will use. 
The authors thank D. Cooper for interesting discussions and useful comments, and B. Fléchelles too, in addition to the fact that he motivated us to finish this paper.

\section{$\SL_2 (\R)$-invariant domains}

As mentioned in the introduction, we let $\SL_2(\R)$ act on the space $V$ of degree 4 homogeneous polynomials with two variables, via the morphism $\rho:\SL_2(\R)\to\SL(V)=\SL_5(\R)$ defined as follows (as in \eqref{irreducible rep})
\begin{align*}
\rho\begin{sm}
a&b\\c&d
\end{sm}\cdot&(\alpha X^4+\beta X^3Y + \gamma X^2Y^2 +\delta XY^3 +\epsilon Y^4) \\
&= \alpha (aX+cY)^4+\beta (aX+cY)^3(bX+dY) + \gamma (aX+cY)^2(bX+dY)^2 \\
&\qquad+\delta (aX+cY)(bX+dY)^3 +\epsilon (bX+dY)^4
\end{align*}

\subsection{The largest invariant properly convex domain}

The $\SL_2(\R)$-action preserves the set of nonnegative polynomials (this is clear from Formula \eqref{irreducible rep}), which is a closed convex cone with nonempty interior that does not contain any line, and hence projects on a properly convex subset of $\PP(V)=\R\PP^4$ whose interior we denote by $\mc O_{\max}$, and whose boundary $\partial\mc O_{\max}$.

\subsection{The limit set}

Note that 
$$
g_t=\rho\begin{pmatrix}
e^t&0\\0&e^{-t}
\end{pmatrix}
$$
is loxodromic with eigenvalues $e^{4t},e^{2t},1,e^{-2t},e^{-4t}$.

This implies $\rho$ maps loxodromic to loxodromic, and the set $\Lambda\subset\R\PP^4$ of eigenlines associated to images of such loxodromics and their dominant eigenvalue is a compact $\rho(\SL_2(\R))$-invariant subset called the limit set.
One can check using  the explicit eigenlines of $g_t$, the $\SL_2(\R)$-action and Formula \eqref{irreducible rep} that 
$$
\Lambda=\{[(aX+bY)^4]: (a,b)\in\R^2\smallsetminus\{(0,0)\}\}.
$$
Hence $\Lambda\subset\partial\mc O_{\max}$.
There is also a dual limit set $\Lambda^*$: the set of hyperplanes spanned by the eigenlines of loxodromics whose eigenvalues are not the smallest one.
One can check that
$$
\Lambda^*=\{\ker \alpha_\theta: \alpha_\theta(P)=P(\cos\theta,\sin\theta),\ \theta\in\R\}.
$$
Then the union of the hyperplanes in $\Lambda^*$ is the set of $[P]$ such that $P(\cos\theta,\sin\theta)=0$ for some $\theta$, i.e.\ such that $P$ is not positive on the whole $\R^2\smallsetminus\{0\}$ or negative on the whole $\R^2\smallsetminus\{0\}$ (by homogeneity and intermediate value theorem).
Hence the union of the hyperplanes in $\Lambda^*$ is exactly the complement of $\mc O_{\max}$.

This implies by work of Benoist \cite{benoist2000automorphismes} that $\mc O_{\max}$ is the largest $\SL_2(\R)$-invariant properly convex open set: any other invariant domain is contained in $\mc O_{\max}$.

\subsection{Other invariant domains}

For instance, the interior $\mc O_{\min}$ of the convex hull of $\Lambda$ is another invariant convex domain, and it is the smallest one (again by work of Benoist).
Then for any $R>0$, the uniform $R$-neighborhood $\mc O_R$ of $\mc O_{\min}$ for the Hilbert metric of $\mc O_{\max}$ is another invariant convex domain.
We are going to see that these are the only ones $\rho (\SL_2 (\R))$-invariant convex domain, that the $\mc O_R$ are round, and that their boundaries foliate $\mc O_{\max}\smallsetminus\mc O_{\min}$ in a convenient way.

\subsection{A fundamental domain for the action of $\SL_2(\R)$ on $\mc O_{\max}$}

We denote by $\Cc$ the closure of $\mc O_{\min}$ in $\mc O_{\max}$.

The point $[(X^2+Y^2)^2]$ is clearly stabilized by $\mr{SO}(2)$, so there is a $\rho$-equivariant embedding $\varphi$ of $\Hb^2$ in $\mc O_{\max}$ defined by $\varphi: g\cdot o \in \Hb^2 \mapsto \rho(g) \cdot [(X^2+Y^2)^2]$, where $o\in \Hb^2$ is the fixed point of $\SO(2)$.

The next lemma exhibits a natural fundamental domain for the action of $\SL_2(\R)$ on $\mc O_{\max}$, which is an interval.

\begin{lemma}\label{lem: sl2orbits param}
	We drop the brackets around polynomials to simplify notations.
\begin{itemize}
	\item Each $\ps2$-orbit of $\mc O_{\max}$ intersects exactly once the interval $I=\left[{(X^2+Y^2)^2},{X^2Y^2}\right)$. The stabilizer of $x \in I$ is $\mr{PSO}(2)$ if $x = {(X^2+Y^2)^2}$, and $S=\{\begin{bsmallmatrix}1&0\\0&1\end{bsmallmatrix},\begin{bsmallmatrix}0&-1\\1&0\end{bsmallmatrix}\}$ otherwise.
	
	\item $\Lambda$ is the set of extremal points of $\overline{\mc C}$.
	
	\item $\PSL_2(\R)$ acts transitively with order 2 stabilisers on $\dni\mc C:=\partial\mc C\smallsetminus\Lambda$ which lies in $\mc O_{\max}$, and each $x\in\dni\mc C$ lies inside a 1-dimensional face of $\mc C$ that is an interval whose extremities are in $\Lambda$.
	Finally, let $PT\mb H^2$ be the circle bundle over $\mb H^2$ where the fiber over $p\in\mb H^2$ is the set of unoriented geodesic lines through $p$.
	Then we have a $\SL_2(\R)$-equivariant homeomorphism  $\phi:\dni\mc C\to PT\mb H^2$ such that for every $x\in\dni\mc C$ with 1-dimensional face $F_x$, and image $\phi(x)=(p,\ell)$, we have $\phi(F_x)=\{(y,\ell): y\in\ell\}$.
	
	\item The segment $I'=\left[{(X^2+Y^2)^2},{(X^2-Y^2)^2}\right)$ intersects the boundary of $\mc C$ at $X^4+Y^4$.
	
	\item The segment $I$ intersects the boundary of $\mc C$ at ${X^4+6X^2Y^2+Y^4}$.
	
	\item For each $R >0$, the action of $\ps{2}$ on $\partial \mc O_R \smallsetminus \Lambda$ is transitive, each point has a stabilizer of order $2$.  
\end{itemize}
\end{lemma}

\begin{proof}
It is elementary to check that the points $X^2 Y^2$, $X^4 +6X^2Y^2+Y^4$, $(X^2+Y^2)^2$, $X^4 + Y^4$ and $(X^2-Y^2)^2$ are aligned in that order.
We denote by $L$ the projective line containing those points.
	
	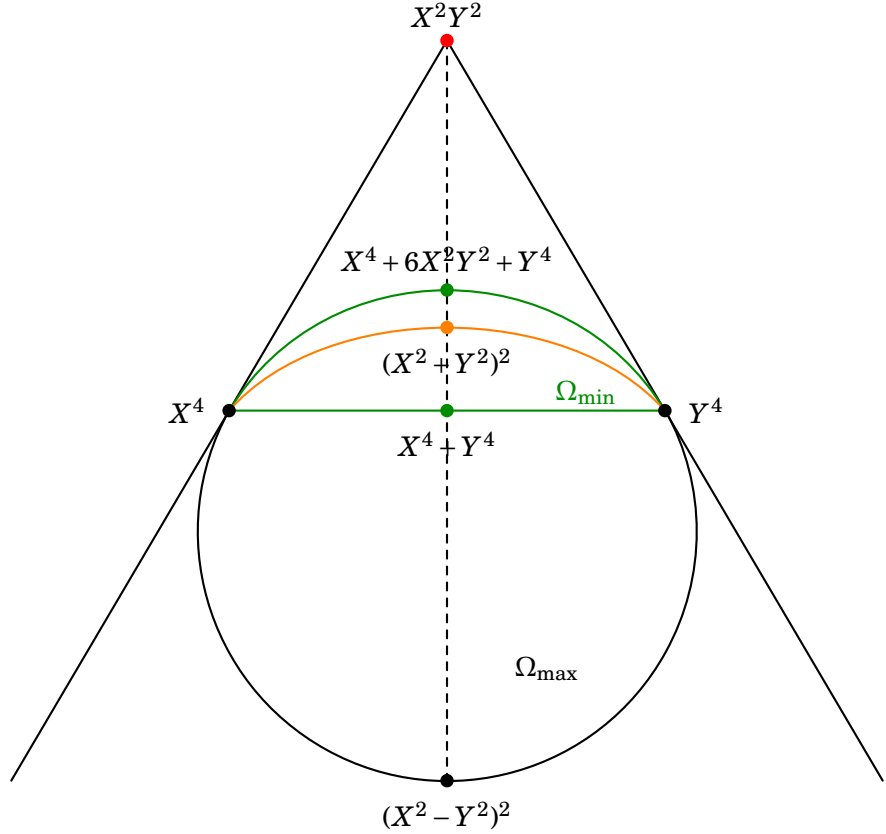
\begin{figure}
		\centering

\begin{tikzpicture}[scale=1.1,
  line cap=round,
  line join=round,
  every node/.style={font=\small}
]

\coordinate (T) at (0,6.00);
\coordinate (L) at (-2.62,1.55);
\coordinate (R) at ( 2.62,1.55);
\coordinate (B) at (0,-2.90);

\coordinate (G) at (0,3.00);
\coordinate (O) at (0,2.55);
\coordinate (M) at (0,1.55);

\draw[black, thick] (T) -- (-5.24,-2.90);
\draw[black, thick] (T) -- ( 5.24,-2.90);

\draw[black, dashed, thick] (T) -- (B);

\draw[black, thick] (L) arc[start angle=151,end angle=389,radius=3.00];

\draw[green!55!black, thick] (L) -- (R);

\draw[green!55!black, thick]
  (L) .. controls (-1.45,3.48) and (1.45,3.48) .. (R);

\draw[orange, thick]
  (L) .. controls (-1.45,2.88) and (1.45,2.88) .. (R);

\fill[red] (T) circle (2.3pt);
\fill[black] (L) circle (2.3pt);
\fill[black] (R) circle (2.3pt);
\fill[black] (B) circle (2.3pt);
\fill[green!55!black] (G) circle (2.3pt);
\fill[orange] (O) circle (2.3pt);
\fill[green!55!black] (M) circle (2.3pt);

\node[above=1pt] at (T) {$X^2Y^2$};
\node[left=5pt] at (L) {$X^4$};
\node[right=5pt] at (R) {$Y^4$};

\node[above=3pt] at (G) {$X^4+6X^2Y^2+Y^4$};
\node[below=4pt] at (O) {$(X^2+Y^2)^2$};
\node[below=4pt] at (M) {$X^4+Y^4$};

\node[green!55!black, right] at (1.18,1.77) {$\Omega_{\min}$};
\node at (1.20,-1.55) {$\Omega_{\max}$};

\node[below=4pt] at (B) {$(X^2-Y^2)^2$};

\end{tikzpicture}

		\caption{Section of $\mc O_{\max}$ and $\mc O_{\min}$}
		\label{fig:fundamental_segment}
	\end{figure}
	
	We first analyze the stabilizer of a standard rotation. Let $r_\theta=\rho\begin{sm}
	\cos\theta&-\sin\theta\\\sin\theta&\cos\theta
	\end{sm}$.
	One can check that the eigenvalues are $e^{4i\theta},e^{2i\theta},1,e^{-2i\theta},e^{-4i\theta}$,
	and that the invariant projective line $L_1\subset\PP(V)$ associated with the eigenvalues $e^{4i\theta},e^{-4i\theta}$ is spanned by $X^4-6X^2Y^2+Y^4$ and $XY^3-X^3Y$, and the invariant projective line $L_2$ associated with the eigenvalues $e^{2i\theta},e^{-2i\theta}$ is spanned by $X^4-Y^4$ and $XY^3+X^3Y$.
	These two lines do not intersect $\overline{\mc O}_{\max}$, otherwise by $\SO(2)$-invariance they would be contained in it and this would contradict proper convexity.
	
	This implies that the only fixed point of $r_\theta$ is $(X^2+Y^2)^2$ if $\theta\not\in \tfrac\pi2\mb Z$.
	If $\theta\in \pi\mb Z$ then $r_\theta=\mr{id}$.
	If $\theta=\tfrac\pi2$ then $r_\theta$ fixes every point of the projective plane $\Pi \subset\PP(V)$  spanned by $L_1$ and $(X^2+Y^2)^2$, which is just
	$$
	\Pi=\Span(X^2Y^2,X^4+Y^4,XY^3-YX^3).
	$$
	In Figure~\ref{fig:fundamental_segment}, $\Pi $ intersects the chosen section in the line $L$, represented dotted.
	
	Let $x\in I$, and $g\in\SL_2(\R)\smallsetminus\{\pm I_2\}$ that fixes it.
	Then $g$ cannot be hyperbolic nor parabolic since $\SL_2(\R)$ acts properly on $\mc O_{\max}$.
	So $g$ is elliptic: there is $h\in \SL_2(\R)$ and $\theta\in\R\smallsetminus \pi\mb Z$ such that $\rho(g)=\rho(h)r_\theta\rho(h)^{-1}$.
	We consider the Cartan decomposition of $h$: there are $\phi,\psi\in\R$ and $t\geq 0$ such that $\rho(h)=r_\phi g_t r_\psi$.
	
	To show that $\Stab_{\PSL_2(\R)}(x)$ is $\PSO(2)$ or $S$ depending on whether $x$ is equal to $(X^2+Y^2)^2$ or not, it suffices to show $t=0$.
	Suppose by contradiction $t>0$.
	Now by the previous analysis, all fixed points of $\rho(g)$ are in $r_\phi g_t \cdot \Pi $, including $x$.
	But $x$ is also in $I$ and hence in $\Pi $.
	Yet, one easily checks that the plane 
	$$g_t \Pi =\Span(X^2Y^2,\ e^{4t}X^4+e^{-4t}Y^4,\ e^{-2t}XY^3-e^{2t}YX^3)$$
	intersects $\Pi $ in exactly $X^2Y^2$.
	Thus  $r_\phi g_t \cdot \Pi $ intersects $\Pi $ in exactly $r_\phi(X^2Y^2)\in\partial\mc O_{\max}$, and $x\not\in\partial\mc O_{\max}$: contradiction!

	We described the stabilisers of points in $I$, let us now prove that every $\SL_2(\R)$-orbit intersects $I$.
	Consider the quotient $M:=(\ps2/S)\times I/\sim$ where $(gS,x)\sim (hS,y)$ if and only if $(gS,x) = (hS,y)$ or [$x=y=(X^2+Y^2)^2$ and $h^{-1}g\in\mr{PSO}(2)$].
	It is a topological manifold of dimension $4$, indeed one can check that by the polar decomposition 
	$$((\ps2/S)\times I/\sim) \simeq 
	((\mathrm{PSO}(2)/S)\times I/\sim) \times \mathrm{Sym}_+^1 (2,\R) 
	\simeq \mathbb{D}^2 \times \mathbb{D}^2,
	$$
	where $\mathrm{Sym}_+^1 (2,\R)$ is the space of positive definite real matrices of size $2$ and determinant $1$. 
	Moreover, we have a natural continuous injection $M\hookrightarrow \mc O_{\max}$ that maps $[gS,x]$ to $gx$, which is therefore an open embedding by the Invariance of Domain Theorem.
	
	We want to prove this map is surjective.
	To do so, it suffices to check that this map is proper.
	Consider  $(g_n)_n\subset\PSL_2(\R)$ and $(x_n)_n\subset I$ such that $(g_nx_n)_n$ is relatively compact in $\mc O_{\max}$.
	Then the  Hilbert distance from $x_n$ to $\mc C$, which is the same as that from $g_nx_n$ to $\mc C$, is bounded.
	Therefore $(x_n)_n$ is relatively compact, and so is $(g_n)_n$ by properness of the action of $\ps2$.
	Thus $(g_nS,x_n)_n$ projects onto a relatively compact sequence in $M$.
	
	Let us now prove that the point of $\Lambda$ are extremal.
	$\overline{\mc C}$ is the convex hull of $\Lambda$, which is compact.
	There must be at least one extremal point, and it must be in $\Lambda$.
	Since $\SL_2(\R)$ acts transitively on $\Lambda$, this implies all points of $\Lambda$ are extremal.
	
	We now prove the third point.
	One can check the interval $(X^4,Y^4)$ lies in $\mc O_{\max}$.
	Since $\SL_2(\R)$ acts doubly transitively on $\Lambda$, this implies all convex combinations of two points of $\Lambda$ are in $\mc O_{\max}$, which implies by convexity all of $\mc C$ (the convex hull of $\Lambda$ minus $\Lambda$ itself) lies in $\mc O_{\max}$.
	Thus $\PSL_2(\R)$ acts properly on $\dni\mc C$, and the orbit map $\PSL_2(\R)/S\to\dni\mc C$ is injective and proper, so surjective by connectedness of $\dni\mc C$ (which is topologically a 3-dimensional sphere minus a circle).
	
	In particular, all faces in $\dni\mc C$ have the same dimension which is either 1, 2 or 3.
	The relative boundary of each is only made of points of $\Lambda$ (if there were other points these would have a face of strictly smaller dimension), concluding the proof of the second point.
	If this dimension was 2 or 3, then the relative boundary is topologically a circle or sphere contained in $\Lambda$, so it would be the entire $\Lambda$, so $\Lambda$ is contained in a projective hyperplane which contradicts the irreducibility of $\PSL_2(\R)$'s action.
	So the faces are 1-dimensional.
	Each face is an interval between two points of $\Lambda$. 
	Since $\SL_2(\R)$ acts doubly transitively on $\Lambda$, all intervals between two points of $\Lambda$ are faces.
	
	The equivariant homeomorphism $\phi:\dni\mc C\to PT\mb H^2$ is defined as follows.
	Let $o\in\mb H^2$ fixed by $\PSO(2)$ and $\ell_0$ the geodesic line through $o$ that is the orbit via $\begin{sm}
	e^t&0\\0&e^{-t}
	\end{sm}$, which is preserved by $s=\begin{bsmallmatrix}0&-1\\1&0\end{bsmallmatrix}$ since $\begin{sm}
	e^t&0\\0&e^{-t}
	\end{sm}\begin{sm}
	0&-1\\1&0
	\end{sm}=\begin{sm}
	0&-1\\1&0
	\end{sm}\begin{sm}
	e^{-t}&0\\0&e^{t}
	\end{sm}$.
	Then for every $g\in\PSL_2(\R)$ set $\phi(\rho(g)(X^4+Y^4))=g\cdot (o,\ell_0)$.
	It is well-defined since $s$ fixes $(o,\ell_0)$.
	Moreover, for any $x\in\dni\mc C$ with face $F_x\subset\dni\mc C$ and image $\phi(x)=(p,\ell)$, we have $\phi(F_x)=\{(q,\ell):q\in\ell\}$.
	Indeed by equivariance it suffices to check it for $x=X^4+Y^4$, for which $F_x=(X^4,Y^4)$ and $\phi(x)=(o,\ell_0)$.
	In this case $F_x$ is the orbit of $x$ under $\rho\begin{sm}
	e^t&0\\0&e^{-t}
	\end{sm}$ and $\{(q,\ell_0):q\in\ell_0\}$ is the orbit of $(o,\ell_0)$ under $\begin{sm}
	e^t&0\\0&e^{-t}
	\end{sm}$, so they identify through $\phi$.
	
	The fourth point comes from the fact that we already know $X^4+Y^4$ is in the boundary of $\mc C$.
	The fifth point comes from the fact that $I=r_{\tfrac\pi4}(I')$ and $r_{\tfrac\pi4}(X^4+Y^4)=X^4+6X^2Y^2+Y^4$.
	The last point comes from that the orbit map $\PSL_2(\R)/S\to\partial\mc O_R\smallsetminus\Lambda$ is injective and proper hence surjective.
\end{proof}

The following lemma is not used in the rest of the paper.
It just serves as an explanation and justification for Figure~\ref{fig:fundamental_segment}.

Let $g_t=\begin{psmallmatrix}e^{t}&0\\0&e^{-t}\end{psmallmatrix}$. If $x\in \mc O_{\max}$, we denote by $g_\R \cdot x$ the orbit of $x$ under the group $(g_t)_{t \in \R}$.

\begin{lemma}\label{lem: sl2orbits param2}
Again we drop the brackets around polynomials to simplify notations.
Let $\Pi$ be the projective plane spanned by $X^4$, $Y^4$ and $X^2Y^2$.
\begin{itemize}
	\item $\Pi \cap \O_{\mathrm{min}}$ is the convex whose boundary is $[X^4,Y^4] \cup g_\R \cdot (X^4+6X^2Y^2 + Y^4)$ (in green on Figure~\ref{fig:fundamental_segment}).
	
	\item $\Pi \cap \O_{\mathrm{max}}$ is the convex whose boundary is $[X^4,X^2Y^2] \cup [X^2Y^2,Y^4] \cup g_\R \cdot {(X^2-Y^2)^2}$ (in black on Figure~\ref{fig:fundamental_segment}).
	
	\item $\Pi \cap \PSL_2(\R)\cdot (X^2+Y^2)^2 = g_\R \cdot (X^2+Y^2)^2$ (in orange on Figure~\ref{fig:fundamental_segment}).
\end{itemize}
\end{lemma}

\begin{proof}
	The first point comes from the fact that these two curves do bound a convex set and are indeed in $\partial (\Pi \cap \O_{\mathrm{min}})$, so they must be its whole boundary.
	
	The second point is proved similarly.
	
	For the third point, note that the $g_{\R}$-orbit of the segment from $X^4+Y^4$ to $X^4+6X^2Y^2+Y^4$ covers the whole $\Pi \cap \mc C$, and the only points whose stabiliser is not order 2 are in the $g_\R$-orbit of $(X^2+Y^2)^2$, so that is $\Pi \cap \PSL_2(\R)\cdot (X^2+Y^2)^2$.
\end{proof}

Finally we recall the argument why the invariant convex domains $\mc O_R$ are round.

\begin{proposition}
	For each $R > 0$, the convex $\mc O_R$ is strictly convex with $\Cc^1$-boundary.
\end{proposition}

\begin{proof}
The $\PSL_2(\R)$-orbit of any point not in $\mc C$ is $\partial\mc O_R\smallsetminus\Lambda$ for some $R$ by the last point of Lemma~\ref{lem: sl2orbits param}.
Thus the $(\mc O_R)_{R \in (0,+\infty)}$ are the only $\PSL_2(\R)$-invariant properly convex domain of $\R\PP^4$ outside of $\mc O_{\min}$ and $\mc O_{\max}$.
Since $\rho$ is conjugated to  its dual $\rho^\ast$, and duality swaps smallest and largest invariant convex domains, we get that there is an involution $\iota: (0,+\infty) \to (0,+\infty)$ such that $\mc O_R^{\ast} \simeq \mc O_{\iota(R)}$.

It is well-known that a properly convex set is strictly convex if and only if its dual has $\Cc^1$ boundary.
So to prove that $\mc O_R$ is strictly convex with $\Cc^1$boundary, it is enough to show that $\mc O_R$ is strictly convex.

If $\mc O_R$ is not strictly convex then there exists $x \in \partial \mc O_R \smallsetminus \Lambda$ such that $x$ is not an extremal point of $\mc O_R$, then all point of $\partial \mc O_R \smallsetminus \Lambda$ are non-extremal by transivity of the $\SL_2(\R)$-action so $\overline{\mc O}_R$ is in the convex hull of $\Lambda$, i.e.\ in $\mc C$, which is absurd.
\end{proof}

\subsection{The stereographic projection of the limit set}\label{sec:stereo}

The limit set $\Lambda$ is the image of the Veronese embedding of $\R\PP^1$ in $\R\PP^4$, namely the image of the map $[a:b] \mapsto (aX+bY)^4$, in other word $\Lambda = h_\R \cdot X^4$.

Taking $(X^4,X^3Y,X^2Y^2,XY^3,Y^4)$ to be our basis of $V_5$, one can observe that the hyperplane $\PP(x_5 = 0)$ is a supporting hyperplane to $\mc O_{\max}$ at $X^4$, thus one gets a map:
\[
\begin{array}{rccc}
\psi : \Lambda \setminus \{X^4\} 
& \longrightarrow 
& \mathbb{R}^3
& =
\mathbb{P}\left( V_5 / \langle X^4 \rangle \right)
\setminus
\mathbb{P}\left( \{x_5 = 0\} / \langle X^4 \rangle \right)
\\[0.4em]
[aX^4+bX^3Y+cX^2Y^2+dXY^3+eY^4]
& \longmapsto
& (\tfrac be,\tfrac ce,\tfrac de)
& = [bX^3Y+cX^2Y^2+dXY^3+eY^4]
\end{array}
\]
Using the parametrisation of $\Lambda \smallsetminus \{ X^4 \}$ given by $\R \ni t \mapsto (tX+Y)^4$, one gets that:
$$
\psi:\R \ni t \mapsto (4t^3,6t^2,4t) \in \R^3 
$$

\noindent The convex hull of $\psi(\R)$ is a parabolic cylinder body, namely the set:
$$
\mathcal{CP} =  \{ (x,y,z) \in \R^3 \, |\, y \geqslant \frac{3}{8} \, z^2 \}
$$

\begin{figure}
		\centering
		\includegraphics[width=0.5\linewidth]{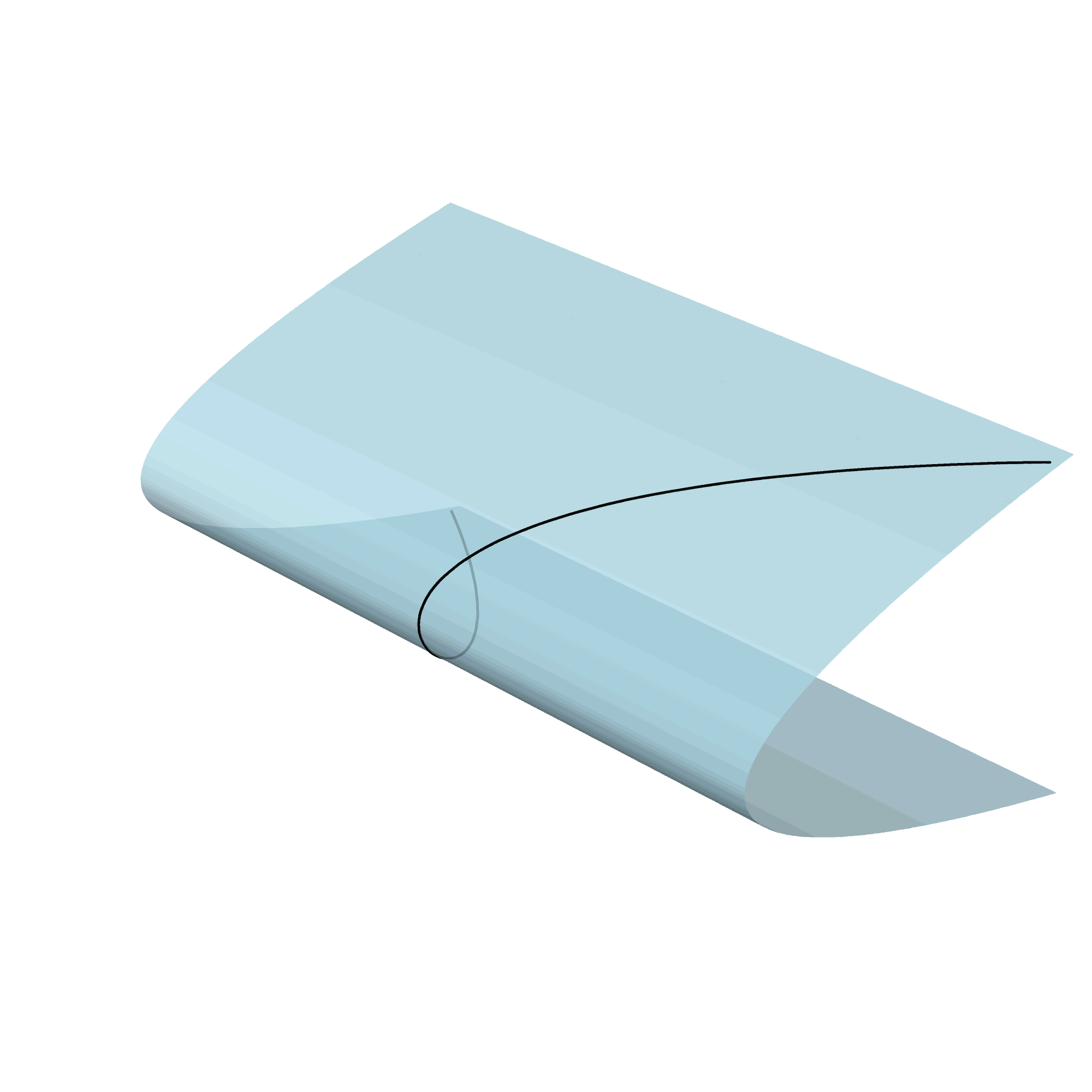}
	\caption{Veronese on the parabolic cylinder}
	\label{fig:parabolic_cylinder}
\end{figure}


\begin{remark}
In \cite{CM2014finitude}, during the proof of Proposition 10.6 and Proposition 10.7, Crampon and the second author claim wrongly that the convex hull of $\psi(\R)$ is $\R^3$ (whereas it is in fact $\mathcal{CP}$ as explained above).
Note that, this error does not break the end of the argument of the proof of Proposition 10.6 and Proposition 10.7 since the action of $h_\Z$ on $\mathcal{CP}$ is also not cocompact.
\end{remark}

\section{The convex core has finite volume and Gromov-hyperbolic universal cover in $\rho(\PSL_2(\R))$-invariant convex domains}

In this very short section we prove point (b) of Theorem~\ref{thm:non_implication}, and hence also its points 1.\ and 2.
The only thing we will need to know about the Hilbert volume on a properly convex open set $\O$ is that it is locally finite and invariant under projective automorphisms.
Thus we postpone until Section~\ref{sec infinite volume} the definition and basic properties of the Hilbert volume.

\begin{proof}[Proof of points (b), 1.\ and 2.\ of Theorem~\ref{thm:non_implication}]
    Let $\mc O$ be a $\SL_2(\R)$-invariant convex domain that contains $\mc C$, so $\mc O$ is $\mc O_{\max}$ or $\mc O_R$ for some $R>0$.
    Let $\mc C'$ be the closed 1-neighborhood of $\mc C$ in $\mc O$; it is $\SL_2(\R)$-invariant.
    
    Using Lemma~\ref{lem: sl2orbits param}, we have a $\SL_2(\R)$-equivariant surjective map $\mc O_{\max}\to\mb H^2$ that maps $I$ onto the fixed point of $\SO(2)$.
    The restriction $\pi:\mc C'\to\mb H^2$ is proper.
    The Hilbert volume $\mu$ from $\mc O$ restricted to $\mc C'$ is a locally finite $\SL_2(\R)$-invariant measure, so its push-forward by $\pi$ is a locally finite $\SL_2(\R)$-invariant measure on $\mb H^2$, which is hence a multiple of the Haar measure.
    
    Let $\G\subset\SL_2(\R)$ be a lattice.
    The Hilbert measure $\mu$ descends to a measure $\mu'$ on $\mc C'/\G$ (which is not the push-forward under the quotient map).
    The push-forward of $\mu'$ under the projection $\mc C'/\G\to\mb H^2/\G$ is a multiple of the Haar measure, which has finite mass since $\G$ is a lattice, so $\mu'$ must also have finite mass.
    
    The Hilbert metric on $\mc C$ is $\SL_2(\R)$-invariant, and proper (closed balls are compact), and $\SL_2(\R)$ acts cocompactly by Lemma~\ref{lem: sl2orbits param}.
    So any cocompact lattice of $\SL_2(\R)$ also acts cocompactly, and we also know such lattice is word-hyperbolic, so $\mc C$ is Gromov-hyperbolic by the Milnor–Švarc Lemma.
\end{proof}

\section{How small can a $\G$-invariant convex domain be?}

Now our goal is to build $\G$-invariant round domains $\O$ where the convex core has infinite volume and is not Gromov-hyperbolic for the Hilbert metric.
For this the idea is to make $\O$ as small as possible: the smaller it is, the bigger the volume of the convex core is (the convex core does not depend on $\O$).
So a first natural question is: how small $\O$ can be?

\subsection{Parametrisation}

To give a quantitative meaning to this, we use a $\SL_2(\R)$-equivariant parametrisation 
$$
\dni\mc C\times[0,1/2)\to\mc O_{\max}\smallsetminus\mc O_{\min};\ ([X^4+Y^4],\alpha) \mapsto [X^4+Y^4-4\alpha X^2Y^2]
$$ 
The first coordinate of this parametrisation records a closest point projection to $\dni\mc C$ and the second coordinate, that we shall call \emph{height} of point and denote $\eta(x)$ for $x\in\mc O_{\max}\smallsetminus\mc O_{\min}$, almost records the distance to $\dni\mc C$, for the Hilbert metric of $\mc O_{\max}$.

\begin{lemma}
    $[X^4+Y^4]$ is a closest point projection of $[X^4+Y^4-4\alpha X^2Y^2]$ for any $0\leq \alpha<1/2$ for the Hilbert metric of $\mc O_{\max}$, and the Hilbert distance is given by $\tfrac12\log\tfrac1{1-2\alpha}$.
    We have the following estimates:
    $$
    \alpha\leq\frac12\log\frac1{1-2\alpha}\leq \frac{\alpha}{1-2\alpha}\sim_{\alpha\to 0}\alpha
    $$
    In other words, for any $x\in\mc O_{\max}\smallsetminus\mc O_{\min}$ we have
    $$
    \eta(x)\leq \dmax (x,\mc C)\leq \frac{\eta(x)}{1-2\eta(x)}\sim_{\eta(x)\to 0}\eta(x)
    $$
\end{lemma}
\begin{proof}
 The formulas and estimates for the distance are just computations, so it suffices to prove $[X^4+Y^4]$ is the closest point projection of $[X^4+Y^4-4\alpha X^2Y^2]$.
 
 Let  $x=[X^4+Y^4]$, $y=[X^4+Y^4-4\alpha X^2Y^2]$, $a=[X^2Y^2]$ and $b=[(X^2-Y^2)^2]$ so that $a,x,y,b$ are aligned in this order and $\dmax (x,y)$ is half the logarithm of the crossratio of these four points.
 The strategy is classical in Hilbert geometry: we find three concurrent hyperplanes $H_x,H_a,H_b$ (their intersection $J$ is a projective plane) containing respectively $x,a,b$ such that $\mc O_{\max}$ is coined between $H_a$ and $H_b$, the convex set $\mc C$ is coined in between $H_a$ and $H_x$ and $y$ is between $H_x$ and $H_b$.
 In Figure~\ref{fig:fundamental_segment}, these hyperplanes intersect the chosen section in horizontal lines.
 Once we find this we can use classical properties of the crossratio.
 Indeed let $x'\in\mc C$ and let us show $\dmax (x',y)\geq \dmax (x,y)$.
 Let $a',b'\in\partial\mc O_{\max}$ such that $a',x',y,b'$ are aligned in this order.
 For each $p=a',x',y,b'$ let $H_p$ be the hyperplane through $p$ and $J$.
 Then the all the hyperplanes $H_a,H_{a'},H_{x'},H_x,H_{y},H_{b'},H_b$ wind around $J$ in this order because of the way we have promised to construct $H_a,H_x,H_b$.
 This makes the crossratio of $H_{a'},H_{x'},H_y,H_{b'}$ bigger than that of $H_a,H_x,H_y,H_b$, which concludes the proof.
 
 We let $H_a=\PP(\langle X^4-Y^4,X^2Y^2, X^3Y,XY^3\rangle)$ and $H_b=\PP(\langle X^4-Y^4,(X^2-Y^2)^2, X^3Y,XY^3\rangle)$ and $H_x=\PP(\langle X^4, Y^4,X^3Y,XY^3\rangle)$, which intersect at the projective plane $\PP(\langle X^4-Y^4,X^3Y,XY^3\rangle)$.
 
 It suffices to check that $H_a$ and $H_b$ are supporting hyperplanes of $\mc O_{\max}$ at respectively $a $ and $b$, and $H_x$ is a supporting hyperplane of $\mc C$ at $x$.

 Let us check $H_x$ is a supporting hyperplane of $\mc C$.
 Recall from Lemma~\ref{lem: sl2orbits param} that the face $F_x$ is an interval between two points of the limit set, here $X^4$ and $Y^4$ ($(X^4,Y^4)$ is clearly in $\mc C$ and does contain $x$).
 Moreover, $\dni\mc C$ is a smooth $\SL_2(\R)$-orbit, so there is a unique supporting hyperplane $T_x\dni\mc C$ at $x$, which is also the unique supporting hyperplane for the entire $F_x$.
 Like $F_x$, the hyperplane $T_x\dni\mc C$ must be $g_t$-invariant.
 Finally $T_x\dni\mc C$ does not contain $X^2Y^2$, otherwise it would meet $\mc O_{\min}$.
 We conclude by observing $H_x$ is the only $g_t$-invariant hyperplane not containing $X^2Y^2$.
 
 Let us check $H_a$ is a supporting hyperplane of $\mc O_{\max}$.
 It is clear that $\PP(\langle X^4,X^2Y^2, X^3Y,XY^3\rangle)$ and $\PP(\langle Y^4,X^2Y^2, X^3Y,XY^3\rangle)$ are supporting hyperplanes, so among the two "segments" of hyperplanes joining them, one is made of supporting hyperplanes, and it is the segment containing $H_a$ (the other contains $\PP(\langle X^4+Y^4,X^2Y^2, X^3Y,XY^3\rangle)$ which is clearly not supporting).
 
 At last, $H_b$ is also a supporting hyperplane of $\mc O_{\max}$ because it is the image of $H_a$ under 
  $r_{\tfrac\pi4}=\rho\begin{sm}
 1/{\sqrt 2}&-1/{\sqrt 2}\\1/{\sqrt 2}&1/{\sqrt 2}
 \end{sm}$.
\end{proof}

For any open convex domain $\O$ such that $\mc C\subset\O\subset\overline\O\smallsetminus\Lambda\subset\mc O_{\max}$, the nonideal boundary $\partial\O\smallsetminus\Lambda$ is parametrised by a continuous function $u_\O:\dni\mc C\to(0,1/2)$.
If $\O$ is $\G$-invariant then so is $u_\O$, and it descends to a function $u_\O:\dni\mc C/\G\to(0,1/2)$ denoted the same way, and called the \emph{height function of $\partial\O$}.

We want to know how small can that function be.
By continuity, it is bounded below by a positive constant on every compact subset of $\dni\mc C/\G$.
If there was a global positive lower bound then $\O$ would contain $\mc O_R$ for some $R$ and the volume of the convex core in $\O$ would smaller than that in $\mc O_R$, hence finite by the previous section, so we want $u_\O$ to be smaller and smaller as we escape all compact sets of $\dni\mc C/\G$, i.e.\ as we go deep in the cusps.

\subsection{First lower bound on the height function}

Convexity of $\O$ gives constraints on the height function: for any $x\in\dni\mc C$, the point $y$ above it at height $u_\O(x)$ is by definition in $\dni\O$, and hence the convex hull of $y$ and $\mc C$ must lie in $\overline\O$, which gives a lower bound on $u_\O(x')$ for $x'$ nearby $x$.
One could say that the height at the point $x$ ``pulls upward'' the height at neighboring points.
More precisely, we have the following.

\begin{lemma}\label{lem functional ineq}
    Let $\O$ be a convex domain such that $\mc C\subset\O\subset\overline\O\smallsetminus\Lambda\subset\mc O_{\max}$.
    Then for any $t\geq 0$ we have
    $$
    u_\O(g_{t/4}[X^4+Y^4])\geq u_\O([X^4+Y^4])e^{-t}
    $$
    In other words, using the action by $\SL_2(\R)$ we can deduce that for any biinfinite geodesic path $c:\R\to\dni\mc C$, for any $t<s\in\R$ we have
    $$
    u_\O(c(t))e^{t-s}\leq u_\O(c(s))\leq u_\O(c(t))e^{s-t}
    $$
    which exactly means $\log u_\O(c(t))$ is $1$-Lipschiptz.
\end{lemma}
\begin{proof}
 Let $u=u_\O([X^4+Y^4])$, so that $[X^4+Y^4-4uX^2Y^2]\in\partial\O$.
 By convexity, taking a convex combination with $[X^4]$, we get $[X^4+e^{-2t}Y^4-4ue^{-2t}X^2Y^2]$ lying in $\overline\O$ above $[X^4+e^{-2t}Y^4]=g_{t/4}[X^4+Y^4]$ (meaning: the closest point projection on $\dni\mc C$ is $[X^4+e^{-2t}Y^4]$).
 It remains to compute at which height it lies, and for this we apply $g_{t/4}^{-1}$, yielding $g_{t/4}^{-1}[X^4+e^{-2t}Y^4-4ue^{-2t}X^2Y^2]=[X^4+Y^4-4ue^{-t}X^2Y^2]$, so the height is $ue^{-t}$, which must be lower than $u_\O(g_{t/4}[X^4+Y^4])$.
\end{proof}

Given a lattice $\G\in\PSL_2(\R)$,
combining the above with the decomposition into a compact part and finitely many cuspidal parts we get the following first lower bound.

Fix $x_0\in\mc C$ and consider the distance function $x\in\mc C\mapsto \dmax (x,\G x_0)$, which descends to a function $f:\mc C/\G\to[0,\infty)$ which we call a \emph{depth function} on $\mc C/\G$.

\begin{proposition}\label{first lower bound}
    Let $\O$ be a $\G$-invariant convex domain such that $\mc C\subset\O\subset\overline\O\smallsetminus\Lambda\subset\mc O_{\max}$.
    Let $f:\mc C/\G\to[0,\infty)$ be a depth function, and $u=u_\O:\dni\mc C/\G\to\R$ the height function of $\O$.
    
    Then there is a constant $C>0$ such that for any $x\in\dni\mc C/\G$ we have 
    $$
    u_{\O}(x)\geq C e^{-f(x)}.
    $$
\end{proposition}
\begin{proof}
 Recall we identified $\dni\mc C$ with $PT\mb H^2$ equivariantly, which gives us a projection $\pi:\dni\mc C\to\mb H^2$.
 Note that if two points $x,y\in\dni\mc C$ are on the same geodesic at distance $t$ for the Hilbert metric of $\mc O_{\max}$, then $\pi(x)$ and $\pi(y)$ are actually at distance $t/4$!
 (To prove this, use the $\SL_2(\R)$ to reduce to the elementary case where $x,y\in[X^4,Y^4]$.)
 
 Decompose $\mb H^2/\G$ into a compact part $K$ containing $\pi(x_0)$ and finitely many cuspidal parts $H_1,\dots,H_n$ that are quotients of horoballs.
 This gives us a decomposition of $\dni\mc C/\G$ into a compact part $\pi^{-1}K$ containing $x_0$ and finitely many cuspidal-like parts $\pi^{-1}H_1,\dots,\pi^{-1}H_n$.
 By compactness of $\pi^{-1}K$ and continuity of $u$, the function $u$ is bounded below on $K$ by a constant $C>0$.
 
 Let $x$ be in a cuspidal part $\pi^{-1}H_i\cap\dni\mc C/\G$.
 Consider a lift $\tilde x\in\dni\mc C$, and its face $F$ in $\dni\mc C$, which is an interval between two points of $\Lambda$, and which is also the concatenation of two antipodal rays from $\tilde x$ that project on two rays of $\dni\mc C/\G$ from $x$.
 Following one of these rays will lead you back to a point $y$ of the compact part $\pi^{-1}K$ in time bounded above by $f(x)$ plus an additive error term, exactly the way this works in the hyperbolic surface with cusps $\mb H^2/\G$, see Fact~\ref{escape the cusp} below.
 
 By Lemma~\ref{lem functional ineq}, we have $u_\O(x)\geq u_\O(y)e^{-\dmax (x,y)}$.
 This concludes the proof as $u_\O(y)\geq C$ and $\dmax (x,y)\leq f(x)+C'$.
\end{proof}

\begin{fact}\label{escape the cusp}
    There is a constant $C>0$ such that for any horoball $H$ of $\mb H^2$, for any geodesic chord $[p,q]\subset H$ with $p,q\in\partial H$, for any $x\in[p,q]$, at least one of $[p,x]$ and $[x,q]$ has length at most $d_{\Hb^2} (x,\partial H)+C$.
\end{fact}

\subsection{Height function and differentiability}

Although we do not give a proof of it here, one can find  a convex set $\O$ such that $\mc C\subset\O\subset\mc O_1$ and $u_\O(x)\leq e^{-f(x)}$ for all $x\in\dni\mc C/\G$.
However, the following result says that such a convex set cannot be round, more precisely its boundary cannot be differentiable at parabolic points.

\begin{lemma}\label{differentiability}
    Let $\O$ be a convex domain such that $\mc C\subset\O\subset\overline\O\smallsetminus\Lambda\subset\mc O_{\max}$, and $p\in\Lambda$.
    Then $p$ is a differentiable point of $\partial\O$ if and only if for any ray $(r(t))_{t\geq0}\subset\dni\mc C$ ending at $p$, we have $u_\O(r(t))e^t\to\infty$ as $t\to\infty$, if and only if there is a ray $(r(t))_{t\geq0}\subset\dni\mc C$ ending at $p$ such that $u_\O(r(t))e^t\to\infty$ as $t\to\infty$.
\end{lemma}
\begin{proof}
 Suppose $\partial\O$ is differentiable at $p$, and take a ray $r(t)\to p$ in $\dni\mc C$.
 Using the action of $\SL_2(\R)$ we can assume $p=[X^4]$ and $r(t)=[e^tX^4+e^{-t}Y^4]$.
 
 Let $u>0$.
 The line through $p$ and $[X^4+Y^4-4uX^2Y^2]$ intersects $\partial\mc O_{\max}$ at two points $p$ and 
 $$q=g_{\tfrac{t_u}4}[(X^2-Y^2)^2]
 =[e^{t_u}X^4+e^{-t_u}Y^4-2X^2Y^2]
 =[X^4+e^{-2t_u}(Y^4-2e^{t_u}X^2Y^2)]$$
 where $4u=2e^{t_u}$.
 By differentiability of $\partial\O$ at $p$, the interval $[q,p]$ must  cross $\dni\O$ at some point of the form 
 $$[X^4+e^{-2t'_u}(Y^4-4uX^2Y^2)]=[e^{t_u'}X^4+e^{-t'_u}Y^4-4ue^{-t'_u}X^2Y^2]=g_{\tfrac{t_u'}4}[X^4+Y^4-4ue^{-t'_u}X^2Y^2]$$
 with $t_u'\geq t_u$, so $t'_u\to\infty$ if $u\to\infty$.
 This point lies above $r(t'_u)$ at height $ue^{-t'_u}$, which means $u_\O(r(t'_u))e^{t'_u}\geq u\to\infty$ as $u\to\infty$.
 As $u\mapsto t_u'$ is increasing continuous, we get $u_\O(r(t))e^t\to\infty$.
 
 Suppose there is a ray $(r(t))_{t\geq0}\subset\dni\mc C$ ending at $p$ such that $u_\O(r(t))e^t\to\infty$ as $t\to\infty$.
 Using the action of $\SL_2(\R)$ we can assume $p=[X^4]$ and $r(t)=g_{\tfrac t4}[X^4+Y^4]$.
 By definition of $u_\O$, the set $\O$ contains $g_{\tfrac t4}[X^4+Y^4-4uX^2Y^2]=[X^4+e^{-2t}(Y^4-4ue^{t}X^2Y^2)]$ for any $u<u_\O(r(t))$.
 Hence it contains $[X^4+e^{-2t}(Y^4-4vX^2Y^2)]$ for any $v<u_\O(r(t))e^t$.
 Since $u_\O(r(t))e^t\to\infty$, this means for every $v\geq 0$ there is $t_v$ such that  $\O$ contains $[X^4+e^{-2t_v}(Y^4-4vX^2Y^2)]$.
 
 Since $T_{p}\partial\mc O_{\max}=\PP(\langle X^4,X^3Y,X^2Y^2,XY^3\rangle)$, to prove $\partial\O$ is differentiable at $p$, it suffices to show that the image $\O'$ of $\O$ under the stereographic projection 
 $$\pi:[aX^4+bX^3Y+cX^2Y^2+dXY^3+eY^4]\mapsto [bX^3Y+cX^2Y^2+dXY^3+eY^4]$$
 is the whole
 	$\mc A:=\{[bX^3Y+cX^2Y^2+dXY^3+Y^4]:b,c,d\in\R\}.$
	
	Since $\mc C\subset\O$, by Section~\ref{sec:stereo} the convex set $\O'$ contains 
	$$E=\pi(\mc C)=\{[xX^3Y+yX^2Y^2+zXY^3+Y^4]:y\geq \tfrac38z^2\}.$$
	We also saw that $\O'$ contains 
	$$F=\{\pi\left([e^{2t_u}X^4-4uX^2Y^2+Y^4]\right)= [-4uX^2Y^2+Y^4]:u\geq0\}.$$
	The convex hull of these two sets $E$ and $F$ is clearly the whole $\mc A$.
\end{proof}

\section{Construction of invariant convex domains with prescribed height function}

In this section, we construct a $\G$-invariant round convex domain (with $\G\subset\PSL_2(\R)$ noncocompact lattice) whose height function is as small as possible given the constraints we found in the previous section.

For this we are going to make the most naive construction: we will consider the desired height function, take its graph, and take the convex hull.
Moreover, we will use the smoothing procedure in the appendix to ensure we get a round domain.

First we need two technical lemmas.

\subsection{Two lemmas}

In Lemma~\ref{lem functional ineq}, we saw that the height function of a convex domain satisfies a sort of functional inequation: the height at a point $x\in\dni\mc C$ "pulls upward" the height at all points on the same biinfinite geodesic, and the strength of this "pull" decreases with the distance to $x$, with exponential speed.
More precisely, if $x,y\in\dni\mc C$ are on the same biinfinite geodesic of $\dni \mc C$, then $u_\O(y)\geq u_\O(x)e^{- \dmax (x,y)}$.

The goal of the following lemmas is to prove that this is in a coarse sense the strongest pull that is happening: given finitely many points $x_1,\dots,x_k\in\dni\mc C$ and heights $\eta_1,\dots,\eta_k$ at these points, consider the smallest height function $u:\dni\mc C\to\R$ above these heights that parametrises the boundary of a convex domain, then $u(x)\leq C\max_i(\eta_ie^{-\dmax (x,x_i)})$ for some constant $C>0$.

\begin{lemma}\label{lem:phipsi}
	Let $\phi,\psi$ be the linear forms on $\mb R_4[X,Y]$ such that for any $P=aX^4+bX^3Y+cX^2Y^2+dXY^3+eY^4$,
	$$\phi(P)=a+e \quad \text{and} \quad \psi(P)=c.$$
	Then for any $\eps_0<2$, there is a constant $C$ such that for all $0\leq\eps\leq \eps_0$ and $g\in \PSL_2(\R)$, we have
	$$\psi(\rho (g)\cdot(X^4-\eps X^2Y^2 + Y^4)) \geq -C \epsilon \Vert g\Vert^{-4} \phi(\rho (g)\cdot(X^4-\eps X^2Y^2 + Y^4)).$$
\end{lemma}

\begin{proof}
	Let $P=X^4-\eps X^2Y^2 + Y^4$ and $g=\begin{psmallmatrix}a&b\\c&d\end{psmallmatrix}$.
	Then
	$$\phi(\rho(g)P)=a^4+b^4+c^4+d^4-\eps a^2b^2 - \eps c^2d^2.$$
	Since $\eps\leq \eps_0<2$, there exists $C$ (which depends only on $\eps_0$ and the choice of $\Vert\cdot\Vert$) such that $\phi(\rho (g)P)\geq C^{-1}\Vert g\Vert^{4}>0$.
	
	On the other hand, since $ad-bc=1$ and hence $a^2d^2+b^2c^2=1+2abcd$, the following ends the proof
	\begin{align*}
		\psi(\rho (g)P) &= 6(a^2c^2+b^2d^2)-\eps(a^2d^2+b^2c^2+4abcd) \\
		&= 6(a^2c^2+b^2d^2) - 6\eps abcd -\eps \\
		&\geq 3\eps(a^2c^2+b^2d^2-2abcd) -\eps \qquad \text{since} \ \eps\leq 2\\
		&\geq -\eps \qedhere
	\end{align*}
\end{proof}

We are now ready to prove the main lemma.

\begin{lemma}\label{lem:dist cvx combi and C}
For every $\epsilon_0<1/2$, there exists a constant $\mr{Cst}>0$ such that for any $x_1,\dots,x_k\in\mc O_{\max}$ with heights $\eta(x_i)<\epsilon_0$, for any convex combination $x$ of $x_1,\dots,x_k$, denoting by $y,y_1,\dots,y_k\in\mc C$ the respective closest point projections,
	$$\eta(x)\leq \mr{Cst}\cdot\max_i  \left(\eta(x_i)\cdot e^{- \dmax (y,y_i)}\right)$$
\end{lemma}

\begin{proof}
%
	Using the action of $\SL_2(\R)$ (Lemma~\ref{lem: sl2orbits param}) we can assume that $x=[X^4-4\eta(x) X^2Y^2+Y^4]$, and hence $y=[X^4+Y^4]$.
	Similarly, for every $i$ there exists $g_i\in\PSL_2(\R)$ such that $x_i=\rho (g_i)\cdot[X^4-4\eta(x_i) X^2Y^2+Y^4]$ and hence $y_i=\rho(g_i)[X^4+Y^4]$.
	
	Let $r,r'\in \mr{SO}(2)$ such that $g_i=r\begin{sm}
	e^{t}&0\\0&e^{-t}
	\end{sm}r'$.
	Since $\rho(r)y=y=\rho(r')y$, one can check that
	$$
	\dmax (y,y_i)= \dmax (y,\rho\begin{sm}
	e^{t}&0\\0&e^{-t}
	\end{sm}y)=4t=4\log\Vert g_i\Vert,
	$$
	where $\Vert g\Vert=\max_{v\in\R^2\smallsetminus\{0\}}\Vert g(v)\Vert/\Vert v\Vert$ (we take the usual Euclidean norm on $\R^2$).
	
	Thus, to conclude the proof it is enough to check that $\eta(x) \leq \mr{Cst} \cdot\max_i  \left(\eta(x_i)\cdot \Vert g_i\Vert^{-4}\right)$.
	
	Let us use the notation from Lemma~\ref{lem:phipsi}: for any $P=aX^4+bX^3Y+cX^2Y^2+dXY^3+eY^4$,
	$$\phi(P)=a+e \quad \text{and} \quad \psi(P)=c.$$
	Then $\eps=-\psi(X^4-\eps X^2Y^2+Y^4)$ and $2=\phi(X^4-\eps X^2Y^2+Y^4)$.
	
	By assumption and since $\{\phi=2\}$ is an affine chart containing $\mc O_{\max}$, the polynomial $X^4-4\eta(x) X^2Y^2+Y^4$ is a convex combination of 
	$$\frac{2\rho (g_i)\cdot(X^4-4\eta(x_i) X^2Y^2+Y^4)}{\phi\left(\rho (g_i)\cdot(X^4-4\eta(x_i) X^2Y^2+Y^4)\right)},\ 1\leq i\leq k.$$
	Thus $4\eta(x)$ is a convex combination of 
	$$\frac{-2\psi\left(\rho(g_i)\cdot(X^4-4\eta(x_i) X^2Y^2+Y^4)\right)}{\phi\left(\rho (g_i)\cdot(X^4-4\eta(x_i) X^2Y^2+Y^4)\right)},\ 1\leq i\leq k:$$
	Each of these terms is, by Lemma~\ref{lem:phipsi}, bounded above by 
	$8C\cdot \eta(x_i)\Vert g_i\Vert^{-4}$.
	This concludes the proof.
\end{proof}

\subsection{Construction of small $\G$-invariant convex domains}

In this section we apply the previously mentioned naive strategy to construct convex domains with prescribed height function.

The consequence of the lemmas of the previous section is the following.

\begin{lemma}\label{lem cvx hull of graph}
    Let $u:\dni\mc C\to(0,\tfrac14)$ be continuous such that $u(p)\leq Cu(q)e^{\dmax (p,q)}$ for all $p,q\in\dni\mc C$, for some constant $C$.
    Let $\O$ be the interior of the convex hull of the graph of $u$.
    Then there is $C'>0$ (that depends on $C$) such that $u(p)\leq u_\O(p)\leq C'u(p)$ for every $p$.
\end{lemma}
\begin{proof}
    Let $p\in\dni\mc C$ and $x\in\partial\O$ above it.
    The fact that $u(p)\leq u_\O(x)$ is obvious from the definition of height function.
    Let us prove the other inequality.
    
    By definition, $x$ is the convex hull of points in the graph of $u$ or in $\Lambda$.
    Approximating $\Lambda$ by the graph of $u$, we can approximate $x$ with points $x'$ as close as we want such that $x'$ is a convex combination of points $y_1,\dots,y_k$ in the graph of $u$.
    These points lie above points $p_1,\dots,p_k\in\dni\mc C$ at height $u(p_i)$.
    
    By Lemma~\ref{lem:dist cvx combi and C}, we have
	$$\eta(x')\leq \mr{Cst}\cdot\max_i  \left(u(p_i)\cdot e^{-\dmax (p',p_i)}\right)$$
	Where $p'$ is the projection of $x'$ on $\dni\mc C$, which is as close as we want to $p$.
	Since 
	$$u(p_i)e^{- \dmax (p',p_i)}\leq Cu(p')$$
	for each $i$ we get
	$$\eta(x')\leq \mr{Cst}\cdot Cu(p').$$
	Letting $x'$ converge to $x$ we get what we want.
\end{proof}

%
%

Now we apply the above principle to a function of the form $u(x)=e^{\chi(f(x))}$ where $f$ is a depth function and $\chi$ is a 1-Lipschitz function that tends to $-\infty$ as $t\to\infty$.

\begin{proposition}\label{prop:construction of convex sets}
Let $\G\subset\PSL_2(\R)$ a noncocompact lattice, $f:\dni\mc C/\G\to\R$ a depth function and $\chi$ a 1-Lipschitz function  bounded above by $-\log100$ and such that $t+\chi(t)$ tends to $+\infty$ as $t\to\infty$.

Then there is a $\G$-invariant convex round domain $\O$ with $\mc C\subset\O\subset\overline\O\smallsetminus\Lambda\subset\mc O_{\max}$ whose height function  $u_\O:\dni\mc C/\G\to\R$ satisfies
$$
C^{-1}e^{\chi(f(x))}\leq u_{\O}(x)\leq C e^{\chi(f(x))}.
$$
for some constant $C>0$.
\end{proposition}

Before diving into the proof, let us remark that it is probably possible to construct a round domain that is much closer to $\mc C$, although we do not need this in the present paper.
Indeed the constraint we found in Lemma~\ref{differentiability} only involves the height function above points in cusps lying on a geodesic line that goes straight into the cusp.
To say this differently, let us identify again $\dni\mc C/\G$ with the circle bundle $PT\mb H^2/\G$, and for each $x$ in a cusp of $\mb H^2/\G$, let $\ell_0(x)$ be the line through $x$ that goes directly in the cusp and $\ell_\theta(x)$ the line through $x$ forming an angle $\theta$ with $\ell_0(x)$.
Lemma~\ref{differentiability} says that for $\O$ to be round we need $u_\O(x,\ell_0(x))e^{f(x)}\to\infty$ as $x$ goes to the end of the cusp, where $f(x)$ is a depth function.
But it does not say anything about $u_\O(x,\ell_{\tfrac{\pi}2}(x))e^{f(x)}$, which can probably stay bounded (but cannot go to zero by Proposition~\ref{first lower bound}).

\begin{proof}
Let $u(x)=e^{\chi(f(x))}=e^{\chi( \dmax (x,\G\cdot x_0)}$ for $x\in\dni\mc C$.
Let us check that it satisfies the assumption of Lemma~\ref{lem cvx hull of graph}.
After taking logarithm, it suffices to prove 
that for all $x,y\in\dni\mc C$, we have $\chi(f(x))\leq \chi(f(y))+ \dmax (x,y)$.
Since $\chi$ is 1-Lipschitz we have 
$\chi(f(x))\leq \chi(f(y))+\vert f(x)-f(y)\vert$.
By triangle inequality we have $\vert f(x)-f(y)\vert\leq \dmax (x,y)$, so we can apply Lemma~\ref{lem cvx hull of graph}.

But we do not want to apply it to $u$ directly, as we want our convex to be round.

Since $\chi\leq -\log100$, the graph of $u$ is contained in a uniform neighbourhood $\mc O_R$ of $\mc C$ in $\mc O_{\max}$.
For each $x\in\dni\mc C$ we denote by $\alpha(x)$ the point above it at height $u(x)$.
For any $x\in\dni\mc C$ let $B_x$ be a closed Hilbert ball of $\mc O_R$ around $\alpha(x)$ of radius $r(x)\leq1$, small enough so that:
\begin{itemize}
    \item $x\mapsto r(x)$ is $\G$-invariant and continuous, hence $x\mapsto B_x$ is $\G$-equivariant and continuous for the Hausdorff topology
    \item every point of $B_x$ is at height at most $2u(x)$ and its projection on $\dni\mc C$ is at distance at most 1 from $x$.
\end{itemize}

Let $x\in\dni\mc C$ and $p\in B_x$, with projection $y\in\dni\mc C$.
Then the height of $p$ is at most $2u(x)$, which is bounded above by $2e^1u(y)$ since $y$ is at distance at most 1 from $x$.
Thus, denoting for each $x\in\dni\mc C$ by $v(x)$ the maximum height of a point above $x$ in $\bigcup_yB_y$, we get that $u(x)\leq v(x)\leq 2e^1u(x)\leq \tfrac14$.

Then $v$ obviously also satisfies the assumptions of Lemma~\ref{lem cvx hull of graph}.
Hence the height function of the convex hull $\O_1$ of the graph of $v$ satisfies $u(x)\leq u_{\O_1}(x)\leq C u(x)$ for some constant.
Note that $\O_1\subset\mc O_R$ by construction.

Since $\mc O_R$ is round, the boundary of $B_x$ is differentiable.
Moreover, $B_x$ converge to $\{\xi\}$ if $x$ converge to $\xi\in\Lambda$.
Thus we can apply  Fact~\ref{cvx hull of smooth is smooth} to  the family of balls $(B_x)_{x\in\dni\mc C}$: $\partial\O_1$ is differentiable at any point of $\partial\O_1\smallsetminus\overline{\mc C}$.
For any $x\in\dni\mc C$, the point $\alpha(x)$ is in the interior of $B_x$, and hence in $\O_1$, so the image of $\alpha$ is in $\O_1$, which implies $\mc C\subset\O_1$.
Thus $\partial\O_1$ is differentiable everywhere outside $\Lambda$.

To prove $\partial\O_1$ is differentiable at $\Lambda$ too we want to use Lemma~\ref{differentiability}.
The height function of $\partial\O_1$ is bounded below by $u(x)$.
Let $p\in\Lambda$ and $r(t)$ a ray in $\dni\mc C$ tending to $p$.

If $p$ is conical, then the projection of $r(t)$ in $\dni\mc C/\G$ passes infinitely often in the compact part $K$, so $u(r(t_n))$ is bounded below by a constant for some $t_n\to\infty$, and so is $u_{\O_1}(r(t_n))$.
Moreover $t\mapsto u_{\O_1}(r(t))e^t$ is increasing by Lemma~\ref{lem functional ineq}, so if it tends to infinity for one sequence $t_n$ then it does for all $t\to\infty$.

If $p$ is parabolic, then $r(t)\in H_i$ for some $i$ after some time, hence $u_{\O_1}(r(t))$ is bounded below by 
$e^{\chi(f(r(t)))}$ with $f(r(t))=t+O(1)$ and hence $\chi(f(r(t)))\geq\chi(t)-O(1)$.
Thus $u_{\O_1}(r(t))e^t\to\infty$ as $t\to\infty$ (recall that  $t+\chi(t)\to\infty$).
We conclude that $\partial\O_1$ is $\mc C^1$.

This implies no nontrivial segment of $\partial\O_1$ can touch $\Lambda$, otherwise it would be in the tangent space at a point of $\Lambda$, hence it would be in $\partial\mc O_R$ (recall $\O_1\subset\mc O_R$), which is impossible since $\mc O_R$ is round.
Hence we can apply Lemma~\ref{lem smoothing}: there exists a $\G$-invariant domain $\O\subset\O_1$ which contains $\alpha(x)$ for any $x$ and such that  $\partial\O\smallsetminus\Lambda$ is Hessian-convex.
Moreover $\partial\O$ is still differentiable at $\Lambda$ for the same reasons as $\O_1$, and it is also strictly convex there because $\O\subset\mc O_R$ and $\mc O_R$ is round.
Thus $\O$ is round.

Since $\alpha(x)\in\O$ for any $x\in\dni\mc C$, we have $u(x)\leq u_\O(x)\leq u_{\O_1}(x)\leq Cu(x)$, which concludes the proof.
\end{proof}

\section{Example where $\mc C$ is not Gromov-hyperbolic}

In this short section, we take care of the second half of point (c) of Theorem~\ref{thm:non_implication}, which is the easier half, because it suffices to take any $\O$ whose boundary has points arbitrarily close to $\mc C$ for the Hilbert metric of $\mc O_{\max}$, and then $\mc C$ is automatically not Gromov-hyperbolic for $\O$'s Hilbert metric.
For $\mc C/\G$ to have infinite volume for $\O$'s Hilbert metric this will not work, we will need the height function of $\O$ to tend to $0$ fast enough in the cusps.

Recall that the thinness of a geodesic triangle in a metric space is the infimum of all $\delta>0$ such that each side of the triangle is contained in the $\delta$-neighborhood of the other two.
The thinness of a geodesic metric space is the supremum of the the thinness of all its geodesic triangles.

\begin{lemma}\label{explosion of thinness}
    Let $(T_n)_n$ be a sequence of triangles of $\R\PP^2$ that converge to a triangle $T\subset\R\PP^2$ and $(\O_n)_n$ a sequence of properly convex open sets of $\R\PP^2$ that converge to a properly convex open set $\O$, such that $T_n\subset \O_n$ for any $n$ and one side $s$ of $T$ is contained in $\partial\O$ and is maximal among segments in $\partial\O$, and the opposite vertex is in $\O$.
    
    Then the thinness of $T_n$ in $\O_n$ tends to infinity.
\end{lemma}
\begin{proof}
    Up to acting by projective transformations we can assume $T_n=T$ for any $n$.
    
    Suppose by contradiction that the thinness stays bounded.
    Let $x$ in the relative interior of $s$.
    For each $n$, let $y_n$ be a closest point to $x$ in the union of the other two sides $s_1\cup s_2$.
    Then $d_{\O_n}(x,y_n)$ is bounded above by some constant.
    Let $a_n,b_n\in\partial\O_n$ such that $a_n,x,y_n,b_n$ are aligned in this order.
    Up to extracting a subsequence we may assume $a_n\to a\in\partial\O$, $b_n\to b\in\partial\O$ and $y_n\to y\in s_1\cup s_2$.
    Note that that $x\neq y$.
    Then the crossratio of $a_n,x,y_n,b_n$ converge to the crossratio of $a,x,y,b$, which is therefore finite.
    So $a\neq x$ and $b\neq y$.
    Since $a,x,b\in\partial\O$, we deduce that the whole $[a,b]$ is in $\partial\O$, including $y$.
    By maximality of $s$ among segments of $\partial\O$, we have $[a,b]\subset s$.
    Since the vertex $v$ opposite to $s$ is in $\O$, the union of the two other sides $s_1\cup s_2$ intersects $\partial\O$ at their endpoint other than $v$, which are endpoint of $s$.
    Since $y\in (s_1\cup s_2)\cap\partial\O$, it is endpoint of $s$ which contains $[a,b]$, so we must have $y=b$ which is a contradiction.
\end{proof}

\begin{proposition}\label{prop not hyperbolic}
  Let $\mc C\subset\O\subset \mc O_{\max}$ be a convex domain whose height function is not bounded below.
  Then $\mc C$ endowed with the Hilbert metric of $\O$ is not Gromov-hyperbolic.
\end{proposition}
\begin{proof}
    Let $x_n\in\dni\mc C$ such that $u_n=u_\O(x_n)\to0$.
    Let $h_n\in\rho(\PSL_2(\R))$ such that $h_nx_n=[X^4+Y^4]$.

    Let $p=[(X^2+Y^2)^2]\in\mc O_{\min}$ and $T_n\subset\mc C$ the triangle with vertices  $p$, $a_n=g_{t_n/4}[X^4+Y^4]$ and $b_n=g_{-t_n/4}[X^4+Y^4]$ where $t_n=-\tfrac12\log u_n\to\infty$.
    Note that $T_n$ converges to the triangle $T$ with vertice $p$, $a=[X^4]$ and $b=[Y^4]$.
    Let us show the thinness of $h_n^{-1}T_n$ in $\O$ tends to infinity.
    
    This thinness is equal to the thinness of $T_n$ in $h_n\O$.
    We have $u_{h_n\O}([X^4+Y^4])=u_n$ and by Lemma~\ref{lem functional ineq} we have
    $$
    u_{h_n\O}([e^tX^4+e^{-t}Y^4])\leq e^{\vert t\vert}u_n \leq \sqrt{u_n}
    $$
    for any $t\in[-t_n,t_n]$.
    Thus the point of $\partial h_n\O$ above $[e^tX^4+e^{-t}Y^4]$ is between $[e^tX^4+e^{-t}Y^4]$ and $[e^tX^4+e^{-t}Y^4-4\sqrt{u_n}X^2Y^2]$, which are at distance at most $C\sqrt{u_n}$ for any fixed Riemannian metric on $\R\PP^4$.

    Let $\Pi$ be the projective plane spanned by $X^4,Y^4,X^2Y^2$, and $\O_n=(h_n\O)\cap \Pi$, which converges after taking a subsequence, to a $\O'$ that contains $[X^4,Y^4]$ in its boundary since by the above the distance from $[e^tX^4+e^{-t}Y^4]$ to $\partial h_n\O$ tends to 0 for any fixed Riemannian metric on $\R\PP^4$.
    Moreover, $[X^4,Y^4]$ is maximal among segments of $\partial\O'$ because $\O'$ is contained in $\Pi\cap\mc O_{\max}$.
    Thus we can apply Lemma~\ref{explosion of thinness} to get that the thinness of $T_n$ in $\O_n$ tends to infinity.
\end{proof}

\section{Example where the convex core has infinite volume}\label{sec infinite volume}

\subsection{Hilbert volumes}

Given a properly convex open set $\O$ of a real projective space, there are several natural $\Aut(\O)$-invariant locally finite measures one can put on it.
We put the $\dim(\O)$-dimensional Hausdorff measure.
In \cite[p.\,3]{Vernicosasymp} one can find more details about Hilbert volumes and other examples of natural measures, which are all comparable because of Benz{\'e}cri's compacness theorem, which we now recall.

\begin{fact}[Benz{\'e}cri's compacness theorem  {\cite[Th.\,2\ p.\,309]{benz_varlocproj}}]\label{benz}
    $\SL_n(\R)$ acts properly and cocompactly on the space of pointed properly convex open subsets of $\R\PP^{n-1}$, endowed with the Hausdorff topology.
\end{fact}

This implies that all balls of a given radius in all properly convex open sets have roughly the same volume.

\begin{corollary}\label{volume of balls}
 For all $n$ and $R>0$ there is a constant $C>0$ such that for any properly convex open set $\O$ and any $x\in\O$ we have
 $$
 C^{-1}\leq \Vol(B_\O(x,R))\leq C
 $$
\end{corollary}

\subsection{Technical estimates on Hilbert volumes}

The goal of this section is to establish the following  lemma.
We will apply it to $\O=\mc O_{\min}$.

\begin{lemma}\label{lem:volume estimate}
 Let $\Omega$ be a properly convex open set of dimension $n$, and fix $o\in\Omega$ and an affine chart containing $\Omega$.
 Fix an open subset $U\subset \dO$ and denote by $U'\subset\O$ the set of rays from $o$ to $U$.
 Suppose $\dO$ is $\mc C^2$ at $U$, and that the Hessian has rank at least $d$ there.
 Then there exists $C>0$ such that for any $\varepsilon>0$,
 \begin{equation*}
  \Vol_{\O_\eps}(U') \geq C^{-1} \tfrac1{\eps^{d/2}},
 \end{equation*}
 where $\O_\varepsilon$ is the $\varepsilon$-neighborhood of $\O$ in the affine chart.
\end{lemma}

The proof is decomposed into three steps.
First we show this volume can be approximated by volumes in $\O$ of Hilbert balls intersected with $U'$, where the radius is $\tfrac12\log\tfrac1\epsilon$.
Then we show that these volumes can be bounded below by volumes in a $d+1$-dimensional section $\PP(W)\cap\O$ of $\O$ so that in this section the boundary has positive Hessian.
Finally we use a result of Colbois--Verovic that gives estimates of volumes of Hilbert balls in convex sets whose boundary has positive Hessian.

\begin{lemma}\label{lem:eps nbhd and balls}
 Let $\Omega$ be a properly convex open set, and fix $o\in\Omega$ and an affine chart containing $\Omega$.
 Fix an open subset $U\subset \dO$ and denote by $U'\subset\O$ the set of rays from $o$ to $U$.
 Then there exists $C>0$ such that for any $\varepsilon>0$ small enough,
 \begin{equation*}
  \Vol_{\O}\left(U'\cap B_\O \left(o,\tfrac12\log(\tfrac1\eps) - C\right)\right) \leq \Vol_{\O_\eps}(U') \leq \Vol_{\O}\left(U'\cap B_\O \left(o,\tfrac12\log(\tfrac1\eps) + C\right)\right),
 \end{equation*}
 where $\O_\varepsilon$ is the $\varepsilon$-neighborhood of $\O$ in the affine chart.
\end{lemma}
\begin{proof}
 We equip the affine chart with a vector space structure by setting $o$ to be the zero vector.
 We claim that there exists a constant $C_1\geq 1$ such that for every $\eps<C_1^{-1}$,
 \begin{equation*}
  (1-C_1^{-1}\eps)^{-1} \O \subset \O_\eps \subset (1-C_1\eps)^{-1}\O.
 \end{equation*}
 In other words, we want to find a constant such that for every $x\in\partial\O$, for every $h>0$, the distance in the affine chart (denoted $d_{\mb A}$) from $(1+h)x$ to $\overline{\O}$ is between $C_1^{-1}h$ and $C_1h$.
 The upper bound is easy: $d_{\mb A}((1+h)x,\overline\O)\leq d_{\mb A}((1+h)x,x)=hd_{\mb A}(0,x)\leq h\max_{y\in\partial\O}d_{\mb A}(0,y)$.
 For the lower bound, let $\alpha>0$ such that $\overline B_{\mb A}(O,\alpha)\subset\O$.
 Then by convexity, for any $x\in\partial\O$ the cone based at $x$ antipodal to the cone through $B_{\mb A}(0,\alpha)$ does not intersect $\O$.
 In particular the ball $B_{\mb A}((1+h)x,h\alpha)$ is in that cone, and hence also does not intersect $\O$, which means $d_{\mb A}((1+h)x,\overline\O)\geq h\alpha$, as desired.
 
 We also claim there exists a constant $C_2\geq 1$ such that for every $\lambda<1$,
 \begin{equation*}
  B_\O(o,\tfrac12\log((1-\lambda)^{-1})) \subset \lambda\O \subset B_\O(o,\tfrac12\log((1-\lambda)^{-1})+C_2).
 \end{equation*}
 This is in fact a simple computation, using the definition of the Hilbert metric.
 Indeed for any unit vector $v$ of the affine chart footed at $0$, let $a,b>0$ such that $-av$ and $bv$ are in $\partial\O$.
 Note that $a$ and $b$ are bounded above and below by constants depending on $\O$.
 Then for any $\lambda<1$, we have
 $$
 d_\O(0,\lambda bv)=\tfrac12\log\frac{\lambda b+a}{(1-\lambda) a},
 $$
 which is at least $\tfrac12\log(1-\lambda)^{-1}$ and at most $\tfrac12\log(C(1-\lambda)^{-1})$ for some constant $C$ independent of $v$ and $\lambda$.
 
 Then we can use the monotonicity property of the Hilbert volume to conclude as follows.
 \begin{align*}
  \Vol_{\O_\eps} U' & \leq \Vol_{(1-C_1^{-1}\eps)^{-1}\O} U' = \Vol_\O ((1-C_1^{-1}\eps)U') = \Vol_\O(U'\cap (1-C_1^{-1}\eps)\O) \\
  & \leq \Vol_\O\left(U'\cap B_\O(o,\tfrac12 \log(\tfrac1\eps) + \tfrac12\log C_1 + C_2)\right),
 \end{align*}
 and
 \begin{align*}
  \Vol_{\O_\eps} U' & \geq \Vol_{(1-C_1\eps)^{-1}\O} U' = \Vol_\O ((1-C_1\eps)U') = \Vol_\O(U'\cap (1-C_1\eps)\O) \\
  & \geq \Vol_\O\left(U'\cap B_\O(o,\tfrac12 \log(\tfrac1\eps) - \tfrac12\log C_1)\right).
 \end{align*}
\end{proof}

\begin{proposition}\label{prop:Hilbert volume and subspaces}
 For any $\eta>0$, there exists a constant $C=C(\eta)>0$ such that for any properly convex open set $\O$ and any subspace $\PP(W)\subset \R\PP^d$ such that $\O':=\O\cap\PP(W)$ is nonempty, for any measurable subset $A\subset \O'$,
 \begin{equation*}
  \Vol_{\O'}(A) \leq C \Vol_\O (A_\eta),
 \end{equation*}
 where $A_\eta$ is the $\eta$-neighborhood of $A$ \emph{in $\O$}.
\end{proposition}
\begin{proof}
 We will need the following fact about Hilbert geometry: for any $\eta>0$ there is $N$ such that for any properly convex open set $\mc O$ of dimension $d-1$, for any $\eta$-separated subset $E\subset\mc O$ (any two points are at distance at least $\eta$), any closed ball of radius $\eta$ of $\mc O$ contains at most $N$ points of $E$.
 Indeed this is a consequence of Benzécri's compactness Theorem~\ref{benz}.
 
 Let us prove the fact.
 Suppose by contradiction that for any $N$ there is a $\eta$-separated set $E_N=\{y_{N,1},\dots,y_{N,k_N}\}$ in a ball $B(x_N,\eta)$ in a properly convex open set $\mc O_N$ such that $k_N\to\infty$.
 By Benzécri's compactness theorem, up to extraction we can assume $\mc O_N\to\mc O$ and $x_N\to x\in\mc O$.
 Then $B(x_N,\eta)\to B(x,\eta)$, and finally $y_{N,i}\to y_i\subset B(x,\eta)$ for any $i$.
 Then $\{y_1,\dots\}$ is an infinite $\eta$-separated subset of $B(x,\eta)$, which is impossible in a compact metric space.

 Consider a maximal $\eta$-separated family $B\subset A$.
 By the above, at most $N$ points of $B$ can be in the same ball of radius $\eta$.
 In other words, any point of $\O$ is in at most $N$ balls centered at points of $B$ of radius $\eta$.
 
 By Corollary~\ref{volume of balls}, there is a constant $C>0$ such that any ball of radius $\eta$ of $\O$ or $\O'$ has measure at least $C^{-1}$ and at most $C$.
 
 We conclude with the following estimates.
 
 \begin{align*}
     \Vol_{\O'}(A) 
     &\leq \Vol_{\O'}\left(\bigcup_{b\in B}B_{\O'}(b,\eta)\right)
     \leq \sum_{b\in B}\Vol_{\O'}\left(B_{\O'}(b,\eta)\right) \\
     &\leq C^2  \sum_{b\in B}\Vol_{\O}\left(B_{\O}(b,\eta)\right)
     \leq NC^2\Vol_\O(A_\eta)
 \end{align*}
 
\end{proof}

\begin{theorem}[{\cite[Th.\,2.2]{ColVerpositivehessian}}]\label{thm:Colbois Verovic}
 Let $\Omega$ be a properly convex open set of dimension $d$, and fix $o\in\Omega$ and an affine chart containing $\Omega$.
 Fix an open subset $U\subset \dO$ and denote by $U'\subset\O$ the set of rays from $o$ to $U$.
 Suppose $\dO$ is $\mc C^2$ with positive definite Hessian at $U$.
 Then there exists a constant $C>0$ such that for any $R>0$,
 \begin{equation*}
  \Vol_\O(B_\O(o,R)\cap U')\geq C^{-1}e^{(d-1)R}.
 \end{equation*}
\end{theorem}
\begin{proof}
 Assume that the whole boundary $\dO$ is $\mc C^2$ with positive definite Hessian.
 Colbois--Verovic  construct a homeomorphism $f:V_1\to V_2$ from an open neighborhood $V_1$ of $\dO$ in $\overline \O$ onto an open neighborhood $V_2$ of $\partial\Hb^d$ in $\overline\Hb{}^d$ which is biLipschitz for the respective Hilbert metrics when restricted to $\O\cap V_1$.
 The estimate $\Vol_\O(B_\O(o,R)\cap U')\geq C^{-1}e^{(d-1)R}$ follows easily in this case.
 
 Let us come back to the general case.
 Consider a relatively compact open subset $V\subset U$,  and denote by $V'\subset\O$ the set of rays from $o$ to $V$.
 One can construct a properly convex open set $\O'$ such that  $\dO'$ is $\mc C^2$ with positive definite Hessian, and $\O\subset\O'$, and $V\subset \dO'$.
 
 Lemma~\ref{lem:eps nbhd and balls} gives $C>0$ such that for every $R$,
 \begin{align*}
  \Vol_\O(B_\O(o,R)\cap U') \geq \Vol_{\O_{e^{-2R-2C}}}(V')
  \geq \Vol_{\O'_{e^{-2R-2C}}}(V')
  \geq \Vol_{\O'}(B_{\O'}(o,R-2C)\cap V'),
 \end{align*}
 which concludes the proof. 
\end{proof}

\begin{proof}[Proof of Lemma~\ref{lem:volume estimate}]
 There exists a subspace $\PP(W)\subset \R\PP^n$ of dimension $d+1$ such that 
 \begin{itemize}
  \item $\O':=\O\cap\PP(W)$ contains $o$;
  \item $\dO'$ intersects $U$;
  \item $\dO'$ is $\mc C^2$ with positive definite Hessian on a relatively compact open subset $V\subset U\cap \dO'$.
 \end{itemize}
 Let $V'\subset \O'$ be the set of ray from $o$ to $V$.
 
 Using relative compactness of $V$ in $U$ and the fact that there exists a constant $\mr{Cst}$ such that $B_\O(x,R)$ is contained in the Euclidean ball centered at $x$ of radius $\mr{Cst}\cdot R$ for all $x\in \O$ and $R>0$, one can find $\eta>0$ such that the $\eta$-neighborhood of $V'$ in $\O$ for the Hilbert metric is contained in the union of $U'$ and a compact set $K$.
 
 Combining Lemma~\ref{lem:eps nbhd and balls}, Theorem~\ref{thm:Colbois Verovic} and Proposition~\ref{prop:Hilbert volume and subspaces}, one obtains constants $C_1,C_2,C_3$ such that for every $\eps$,
 \begin{align*}
  \Vol_{\O_\eps}(U') & \geq \Vol_{\O}\left(U'\cap B_\O \left(o,\tfrac12\log(\tfrac1\eps) - C_1\right)\right) \\
  & \geq C_2^{-1}\Vol_{\O'}\left(V'\cap B_{\O'} \left(o,\tfrac12\log(\tfrac1\eps) - C_1 -\eta\right)\right) - \Vol_\O(K) \\
  & \geq C_2^{-1} C_3^{-1} e^{d(\tfrac12\log(\tfrac1\eps) - C_1 -\eta)} - \Vol_\O(K) \\
  & \geq C^{-1} \tfrac1{\eps^{d/2}}\qedhere
 \end{align*}
\end{proof}

\subsection{Construction of infinite volume weakly geometrically finite examples}

\begin{proposition}\label{prop infinite volume}
 Consider a nonuniform lattice $\Gamma\subset \PSL_2(\R)$.
 Let $f:\mc C/\G\to[0,\infty)$ be a depth function.
 
 Consider a $\Gamma$-invariant round convex domain $\O$ whose height function satisfies $u_\O(x)\leq Ce^{-\frac12f(x)}$ for some constant $C>0$.
 Then $\mc C/\Gamma$ has infinite volume for the Hilbert volume of $\O$.
\end{proposition}

\begin{proof}
Recall we have a proper $\SL_2(\R)$-equivariant map $\psi:\mc C\to\mb H^2$.
Consider a decomposition of $\mb H^2/\G$ into a compact part and finitely many cusps, and let $H$ be the preimage under $\psi$ of a horoball corresponding to one of these cusps, stabilised by a parabolic group $P\subset \G$.
We want to show $H/P$ has infinite Hilbert volume (in $\O$).

Denote $h_s:=\begin{psmallmatrix}1&s\\0&1\end{psmallmatrix}$ and $g_t:=\begin{psmallmatrix}e^{t}&0\\0&e^{-t}\end{psmallmatrix}$.
Without loss of generality, we may assume $P=\{h_l:l\in\Z\}\subset\Gamma$ and $g_tH\subset H$ for any $t\geq 0$.
Let $A=H\smallsetminus g_1H$, so that $H/P=\bigsqcup_{k\geq 0}(g_kA/P)$.
We are going to bounded from below the volume of $g_kA/P$ independently of $k$.
 
 Fix an affine chart containing $\overline{\mc O}_{\max}$, with a $\rho(\mr{SO}(2))$-invariant Euclidean metric.
Consider $o$ in the interior of $\mc C$ and in $A$ and a relatively compact open subset $U$ of $\dni\mc C$ such that the set of rays $U'\subset\mc C$ from $o$ to $U$ is contained in $A$, see Figure~\ref{fig:infinite volume}

	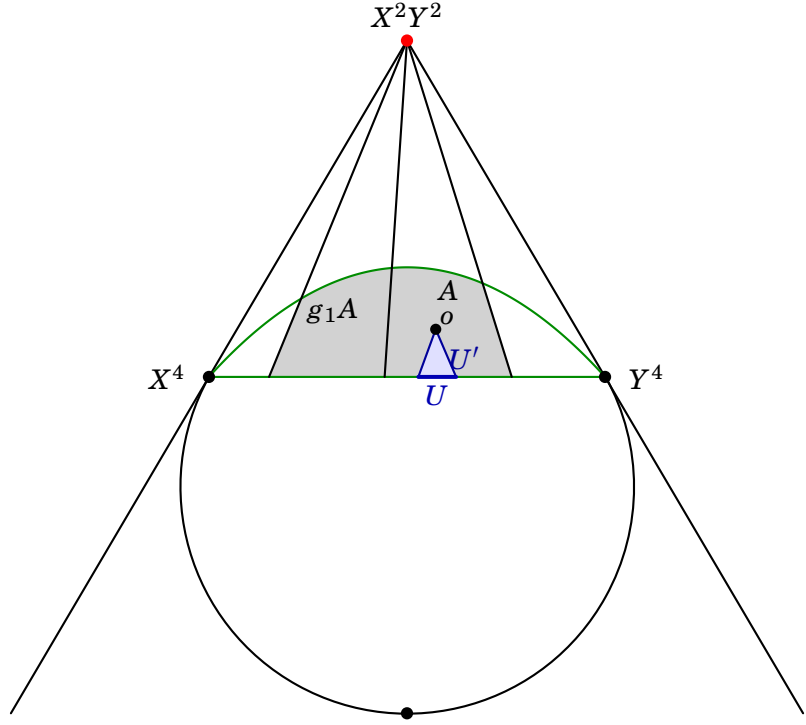
\begin{figure}
		\centering

\begin{tikzpicture}[
  line cap=round,
  line join=round,
  every node/.style={font=\small}
]

\coordinate (T) at (0,6.00);
\coordinate (L) at (-2.62,1.55);
\coordinate (R) at ( 2.62,1.55);
\coordinate (B) at (0,-2.90);

\def\a{3.00}
\def\b{0.2112347765}

\coordinate (P1) at (-1.8248,1.55);
\coordinate (P2) at (-0.2958,1.55);
\coordinate (P3) at ( 1.3858,1.55);

\coordinate (U1) at (0.15,1.55);
\coordinate (U2) at (0.65,1.55);

\coordinate (o) at (0.38,2.18);

\fill[gray!35]
  plot[domain=-1.4:-0.2, samples=50] (\x,{\a-\b*\x*\x}) -- (P2) -- (P1) -- cycle;

\fill[gray!35]
  plot[domain=-0.2:1.0, samples=50] (\x,{\a-\b*\x*\x}) -- (P3) -- (P2) -- cycle;

\fill[blue!12] (o) -- (U1) -- (U2) -- cycle;
\draw[blue!60!black, thick] (o) -- (U1);
\draw[blue!60!black, thick] (o) -- (U2);

\draw[black, thick] (T) -- (-5.24,-2.90);
\draw[black, thick] (T) -- ( 5.24,-2.90);

\draw[black, thick] (L) arc[start angle=151,end angle=389,radius=3.00];

\draw[green!55!black, thick] (L) -- (R);
\draw[green!55!black, thick] plot[domain=-2.62:2.62, samples=120] (\x,{\a-\b*\x*\x});

\draw[black, thick] (T) -- (P1);
\draw[black, thick] (T) -- (P2);
\draw[black, thick] (T) -- (P3);

\draw[blue!70!black, line width=1.5pt] (U1) -- (U2);

\fill[red] (T) circle (2.3pt);
\fill[black] (L) circle (2.3pt);
\fill[black] (R) circle (2.3pt);
\fill[black] (B) circle (2.3pt);
\fill[black] (o) circle (2.1pt);

\node[above=1pt] at (T) {$X^2Y^2$};
\node[left=5pt]  at (L) {$X^4$};
\node[right=5pt] at (R) {$Y^4$};

\node at (-0.98,2.43) {$g_1A$};
\node at (0.55,2.7) {$A$};
\node at ($(o)+(0.15,0.15)$) {$o$};
\node[blue!70!black] at ($(U1)!0.5!(U2)+(0,-0.25)$) {$U$};
\node[blue!60!black] at (0.76,1.8) {$U'$};

\end{tikzpicture}

		\caption{Illustration of $A$, $U$ and $U'$ in a slice for the proof of Proposition~\ref{prop infinite volume}}
		\label{fig:infinite volume}
	\end{figure}

Observe that, if $U$ is small enough and $o$ close enough to $U$ then for all $s\in \R$, that $h_s U' \cap U' \neq \varnothing$ implies $|s| \leqslant 1$.

Let $\pi:\mc C \to \mc C/\G$ be the projection map on the quotient. For every measurable $E \subset \O$ relatively compact, one has:
$$
\Vol_{\O/\G} (\pi(E) ) = \int_E \frac{1}{\# (\G x \cap E)} d\Vol_{\O}(x).
$$

Hence, the sets $(\pi (g_k U'))_k$ are disjoints, hence $\Vol_{\Omega/\Gamma} (\mc C/\G) \geqslant \sum_{k \geqslant 1} \Vol_{\O/\G} (\pi(g_k U'))$. To conclude the proof, it is enough to show that the volumes in $\O/\G$ of the $\pi(g_k U')$ are bounded below.
 
 \medskip
 
Using the fact that $g_{-t}h_{s}g_t=h_{e^{-2t}s}$, and the fact that $h_s U' \cap U' \neq \varnothing$ implies $|s| \leqslant 1$, one gets that:
$$
\forall x \in g_k U', \quad \# (\G x \cap g_k U') \leqslant 2e^{2k} + 1.
$$
And hence
$$
\Vol_{\O/\G} (\pi(g_k U')) \geqslant \frac{e^{-2k}}{3} \Vol_{\O}(g_k U') =  \frac{e^{-2k}}{3} \Vol_{g_{-k}\O}(U').
$$

 Thus to conclude the proof it suffices to find a constant $C_1>0$ such that for every $k$,
 $$ \Vol_{g_{-k}\O}(U') \geq C_1^{-1}e^{2k}.$$
 
It suffices to find $C_2$ such that $g_{-k}\O$ is contained in the Euclidean $C_2 e^{-2k}$-open neighborhood $\O'$ of $\mc C$ for every $k$. Indeed, in this case
$$
\Vol_{g_{-k}\O}(U') \geqslant \Vol_{\O'}(U') \geqslant \mathrm{Cst}^{-1}\frac{1}{(C_2 e^{-2k})^{\nicefrac{2}{2}}} = (\mathrm{Cst}\cdot C_2)^{-1} e^{2k}
$$
by Lemma~\ref{lem:volume estimate}, since $\partial \mc C$ is $\mc C^2$ at $U$ and the Hessian has rank at least $2$. Note that we have used the decreasingness of $\Omega \mapsto \Vol_\O (A)$ for the first inequality.

In order to find $C_2$, we use Lemma~\ref{lem: dim4ex: eucl diam of hilb balls} below: there is a constant $C_3$ such that for every $x\in\dni\mc C$, the point of $g_{-k}\partial \O$ above it, which is at height 
$$u_{g_{-k}\O}(x)
=u_\O(g_kx)
\leq Ce^{-\tfrac12f(g_kx)},$$
is at Euclidean distance at most $CC_3e^{-\tfrac12f(g_kx)-d(o,x)}$ from $x$ and hence from $\mc C$.
To conclude the proof it suffices to prove $-\tfrac12f(g_kx)-d(o,x)\leq -2k+\mr{Cst}$, so it suffices to show $f(g_kx)+d(o,x)\geq 4k-\mr{Cst}$

Up to an additive error term independent of $k$, we have $f(g_kx)= \dmax (g_kx,\G\cdot o)$.
Note that $\mc C=\bigsqcup_{\ell\in\mb Z}g_\ell A$, and for any $\ell\in\mb Z$ and $y\in g_\ell A$, we have $d(o,y)\geq 4\vert \ell\vert-\mr{Cst}$ and if $\ell\geq 0$ then $d(y,\G\cdot o)\geq 4\ell-\mr{Cst}$.

Let $\ell\in\mb Z$ such that $x\in g_\ell A$.
If $\vert \ell\vert\geq k$ then we have $f(g_kx)+ \dmax (o,x)\geq 4k-\mr{Cst}$ as desired.
Suppose $\vert \ell\vert\leq k$.
Then $g_kx\in g_{k+\ell}A$ with $k+\ell\geq 0$ so $f(g_kx)\geq 4(k+\ell)-\mr{Cst}$, and $f(g_kx)+ \dmax (o,x)\geq 4(k+\ell)+4\vert\ell\vert -\mr{Cst}\geq 4k-\mr{Cst}$ as desired, which concludes the proof.
\end{proof}

 \begin{lemma}\label{lem: dim4ex: eucl diam of hilb balls}
  There is a constant $C_3$ such that for all $x\in\dni\mc C$ and $0\leq h\leq 1/2$, the point at height $h$ above $x$ is at Euclidean distance at most $C_3 h e^{- \dmax (o,x)}$ from $x$.
 \end{lemma}
 \begin{proof}
     There is $g\in\rho(\SL_2(\R))$ such that $x=g[X^4+Y^4]$.
     Since the Euclidean metric is $\rho(\mr{SO}(2))$-invariant, using Cartan decomposition we may assume $g=g_{t/4}$ where $t= \dmax ([X^4+Y^4])$.
     
     Thus $x=g_{t/4}[X^4+Y^4]=[X^4+e^{-2t}Y^4]$ and the point $y$ at height $h$ above $x$ is $y=g_{t/4}[X^4+Y^4-4hX^2Y^2]=[X^4+e^{-2t}Y^4-4he^{-t}X^2Y^2]$, which is clearly at Euclidean distance at most a constant times $he^{-t}$ from $x$, and from $[X^4]$ too actually, since the affine chart contains the closure of $\mc O_{\max}$.
 \end{proof}
 
 \subsection{Proof of points (c), 3.\ and 4.\ of Theorem~\ref{thm:non_implication}}

The point (c) of Theorem~\ref{thm:non_implication} is an immediate consequence of Propositions~\ref{prop not hyperbolic}, \ref{prop infinite volume} and \ref{prop:construction of convex sets}.

\appendix

\section{Smoothing convex sets}

The goal is to check that the smoothing procedure of Danciger--Guéritaud--Kassel \cite[\S 9]{DGK17} extends to certain noncompact convex projective manifolds.
We will also check that their technique can be combined with Cooper--Long--Tillman's \cite[\S8]{CLT18} to produce a manifold whose boundary Hessian-convex, in the sense smooth with positive Hessian.
Unfortunately, the general result \cite[Prop.\,8.3]{CLT18} of Cooper--Long--Tillman that we wanted to use is not true, a counterexample being the tetrahedron with two vertices removed.
After an email exchange with Cooper, we learned that their result probably works under the additionnal assumptions that totally geodesic pieces in the boundary of the universal cover are compact.
This property is indeed verified in \cite[\S 9]{DGK17} and our case below, and plays an important role.

The construction splits into two steps: first making a convex set with differentiable boundary by taking convex hulls of convex sets with differentiable boundaries, and then sand this down to something Hessian-convex.

The following is a generalisation of the fact that given finitely many convex sets with $\mc C^1$ boundaries, their convex hull also has $\mc C^1$ boundary.

\begin{fact}\label{cvx hull of smooth is smooth}
 Let $\mc B$ be a family of compact convex subsets of $\R^n$ with nonempty interior and $\mc C^1$ boundaries, and $\Lambda\subset \R^n$ compact such that $\Lambda\cup\bigcup_{B\in\mc B}B$ is compact.
 Let $\mc C_1$ be the convex hull of $\Lambda$ and $\mc C_2$ the convex hull of $\Lambda\cup\bigcup_{B\in\mc B}B$.
 
 Then $\partial\mc C_2$ is differentiable at any point $x\in(\partial\mc C_2)\smallsetminus\mc C_1$.
\end{fact}
\begin{proof}
Consider two supporting hyperplanes $H,H'$ of $\partial\mc C_2$ at $x$, and let us show $H=H'$.

Let $F$ be the closed face of $\partial \mc C_2$ containing $x$.
Then $H$ and $H'$ are supporting hyperplanes at any point of $F$.
Since $x\not\in\mc C_1$, there must be a point $y\in F$ which is extremal for $\mc C_2$ and is not in $\mc C_1$.
In particular $y\in\bigcup_{B\in\mc B}B$, so $y\in B$ for some $B\in\mc B$, and in fact we must have $y\in\partial B$.
As we said $H$ and $H'$ are supporting hyperplanes of $\mc C_2$ at $y$, and hence they are also supporting hyperplanes of $B$ at $y$ (since $B\subset\mc C_2$).
Since $\partial B$ is differentiable at $y$, we conclude that $H=H'$.
\end{proof}

\begin{lemma}\label{lem smoothing}
    Let $\mc C\subset\R\PP^n$ be a locally closed (open in closure) properly convex set with nonempty interior and $\Lambda=\overline{\mc C}\smallsetminus\mc C$, such that $\partial \mc C$ is differentiable at any point outside of $\Lambda$, and such that any segment of $\partial\mc C$ that intersects $\Lambda$ is contained in it.
    Let $\Gamma$ be a discrete group of projective transformations that preserves $\mc C$ and acts properly discontinuously on it.
    
    Then for any $\Gamma$-invariant closed subset $A$ in the interior of $\mc C$, there exists $A\subset\mc C'\subset\mc C$ $\Gamma$-invariant such that $\partial \mc C'$ is Hessian convex at any point outside $\Lambda$.
\end{lemma}
The proof is a combination of ideas from \cite[\S 9]{DGK17} and a local smoothing procedure of \cite[\S8]{CLT18} that we reprove because we very slightly changed it, and checked it also applies in cases with torsion.

\begin{fact}[{\cite[\S8]{CLT18}}]
Let $\epsilon>0$ and $N_\epsilon:\R\to\R$ a smooth, convex and even function such that $N_\epsilon(x)=\vert x\vert$ if $\vert x\vert\geq \epsilon$.
Set $M_\epsilon(x,y)=\tfrac12(x+y+N_\epsilon(x-y))$, which is a smooth convex function $\R^2\to\R$ such that
\begin{itemize}
    \item $M_\epsilon(x,y)=M_\epsilon(y,x)$ and $M_\epsilon(x+t,y+t)=M_\epsilon(x,y)+t$ for all $x,y,t$
    \item $M_\epsilon(x,y)=\max(x,y)$ if $\vert x- y\vert\geq\epsilon$
    \item $\partial_xM_\epsilon$ and $\partial_yM_\epsilon$ are nonnegative and their sum is 1.
    \item $x\leq x'$ and $y\leq y'$ imply $M_\epsilon(x,y)\leq M_\epsilon(x',y')$
\end{itemize}
As a consequence, if $f,g:U\subset\R^n\to\R$ are two convex functions then
\begin{itemize}
    \item $M_\epsilon\circ(f,g):U\to\R$ is convex
    \item if $f$ and $g$ are smooth with positive Hessian then so is $M_\epsilon\circ(f,g)$.
\end{itemize}
\end{fact}
\begin{proof}
Smoothness and convexity of $M_\epsilon$ follow immediately from smoothness and convexity of $N_\epsilon$.
The first three points about $M_\epsilon$ are immediate computations.
The last point about $M_\epsilon$ is a consequence of the third.

Let us check that $M_\epsilon\circ(f,g)$ is convex.
Let $p,q\in\R^n$ and $0\leq t\leq 1$.
By convexity of $f$ and $g$ we have $f(tp+(1-t)q)\leq tf(p)+(1-t)f(q)$ and $g(tp+(1-t)q)\leq tg(p)+(1-t)g(q)$, and hence by $M_\epsilon$'s last property it follows that
\begin{align*}
M_\epsilon(f(tp+(1-t)q),g(tp+(1-t)q)) 
&\leq M_\epsilon( tf(p)+(1-t)f(q), tg(p)+(1-t)g(q)) \\
&=M_\epsilon(t(f(p),g(p))+(1-t)(f(q),g(q)))\\
&\leq tM_\epsilon(f(p),g(p))+(1-t)M_\epsilon(f(q),g(q))
\end{align*}
where the last inequality comes from $M_\epsilon$'s convexity.

Finally, suppose $f,g:U\subset\R^n\to\R$ are smooth with positive Hessian, and let us check $h=M_\epsilon\circ(f,g)$ has positive Hessian.
For any vector $v\in\R^n$, we have
$$
\mr{Hess}(h)(v,v) = \partial_xM_\epsilon\mr{Hess}(f)(v,v)+\partial_yM_\epsilon\mr{Hess}(g)(v,v) + \mr{Hess}(M_\epsilon)((df(v),dg(v)),(df(v),dg(v)))
$$
The second term is nonnegative since $M_\epsilon$ is convex.
Finally, the first term is positive because $\mr{Hess}(f)(v,v)$ and $\mr{Hess}(g)(v,v)$ are positive and $\partial_xM_\epsilon$ and $\partial_yM_\epsilon$ are nonnegative and their sum is 1.
\end{proof}

\begin{fact}[{\cite[\S8]{CLT18}}]\label{local smoothing}
Let $G$ be a subgroup of $\mr O(n)$, and $K\subset\R^n$ a $G$-invariant convex compact subset with $0$ in its interior, and $f:K\to\R$ a $G$-invariant convex function which is zero on $\partial K$ and negative in the interior.
Let $\min f<t<s<0$.

Then there exists $g:K\to\R$ a convex $G$-invariant function such that
\begin{itemize}
    \item $f\leq g$ with equality on $\{f\geq s\}$
    \item if $f$ is Hessian-convex on $U\subset K$, then $g$ is Hessian-convex on $U\cup \{f<t\}$.
\end{itemize}
\end{fact}
\begin{proof}
 Let $\epsilon=(s-t)/4$.
 Take $h:K\to\R$ a Hessian-convex $\mr O(n)$-invariant function of the form $h(x)=a\Vert x\Vert^2-(s+t)/2$ with $a$ small enough so that $t+\epsilon<h(x)<s-\epsilon$ on $K$.
 Then $g=M_\epsilon\circ(f,h)$ satisfies the conclusion of the fact.
\end{proof}

Given a compact convex set $\mc C$ with nonempty interior, a \emph{cap} is a connected component of $\mc C\smallsetminus H$ where $H$ is a hyperplane going through the interior of $\mc C$.

\begin{corollary}\label{cor local smoothing}
    Let $\mc C$ be a locally closed properly convex set with nonempty interior and $\Lambda=\overline{\mc C}\smallsetminus\mc C$.
    Let $\Gamma$ be a discrete group of projective transformations that preserves $\mc C$ and acts properly discontinuously on it, and $G\subset \Gamma$ finite.
    Let $D\subset\mc C$ be a $G$-invariant cap such that $D\cap \gamma D=\emptyset$ for any $\gamma\in\Gamma\smallsetminus G$.
    Finally let $U\subset \partial\mc C\cap D$ be compact.
    
    Then there is a $\Gamma$-invariant convex subset $\mc C'\subset\mc C$ and a $\Gamma$-equivariant homeomorphism $\phi:\mc C\to\mc C'$ such that 
    \begin{itemize}
        \item $\phi(x)=x$ for any $x\not\in\Gamma\cdot D$;
        \item $\partial\mc C'$ is Hessian convex at any point of $\Gamma\cdot \phi(U)$;
        \item if $\partial\mc C$ is Hessian convex at $x\not\in\Lambda$ then $\partial\mc C'$ is Hessian convex at $\phi(x)$.
    \end{itemize}
\end{corollary}
\begin{proof}
    Let $H$ be the hyperplane through the interior of $\mc C$ such that $D$ is a component of $\mc C\smallsetminus H$.
    Since $D$ is $G$-invariant, so are $H$ and $H\cap \mc C$ (the base of the cap).
    Now we apply the following standard fact.
    
    \begin{fact}\label{fact fixed point}
        If a finite group of projective transformation $G'$ preserves a subset $A\subset\R\PP^n$ which is contained in some affine chart and convex (but not necessarily bounded), then $G'$ fixes a point of the relative interior of $A$.
        
        If moreover $A$ is bounded in the affine chart, then $G'$ preserves a hyperplane that does not intersect $\bar A$.
    \end{fact}
    \begin{proof}
        Let $x$ be in the relative interior of $A$.
        Then $G'\cdot x$ is a finite $G'$-invariant set, and its convex hull is a compact convex $G'$-invariant subset $K$ of the relative interior of $A$.
        We can reiterate this: let $y\in\mr{int}(K)$ and $K'$ the convex hull of $G'\cdot y$ which is a compact convex $G'$-invariant subset of the properly convex set $K$.
        Then any center of mass of $K'$ coming from the Hilbert geometry of $K$ (see e.g.\ \cite[Lem.\,4.2]{MarquisHandbook}) is a fixed point for $G'$.
        
        To construct the invariant hyperplane, consider the action of $G'$ on the set of hyperplanes, which is the projective space of the dual of $\R^{n+1}$.
        There the set of hyperplanes not intersecting $\bar A$ is convex and contained in an affine chart, for instance the affine chart of all hyperplanes not containing a given point $a\in A$.
        Thus we can apply the first part of the fact.
    \end{proof}
    
    By the above Fact~\ref{fact fixed point}, as $G$ is finite it fixes a point $x$ in the relative interior of $H\cap \mc C$ and preserves a codimension 1 subspace $L$ of $H$ that does not intersect $H\cap \mc C$.
    Moreover, it is well known that $G$ must preserve an inner product, and hence fixes the point $y\in\R\PP^n\smallsetminus H$ that corresponds to the orthogonal of $H$.
    
    Note that every point of the line spanned by $x$ and $y$ is fixed by $G$ (indeed $G$ is finite so each of its elements acts on the line either by the identity or a reflexion, but here reflexions are prohibited as otherwise $G$ could not preserve the cap $D$).
    Take such a point $z$ which is in the interior of $\mc C$ but not in $\bar D$.
    The hyperplane $H'$ through $z$ and $L$ is $G$-invariant.
    
    Now there is a $G$-invariant affine chart $\psi:\mb R^n=\mb R^{n-1}\times\mb R\hookrightarrow\R\PP^n$ such that
    \begin{itemize}
        \item its image is $\R\PP^n\smallsetminus H'$;
        \item $z$ corresponds to the point at infinity in the vertical direction $\{0\}\times\R$;
        \item  $\mb R^{n-1}\times\{0\}$ maps to $H\smallsetminus L$
        \item $G$ acts on the affine chart by euclidean isometries;
        \item $\bar D$ is contained in this affine chart, and $\psi^{-1}(\partial\mc C\cap \bar D)$ is the graph of a nonpositive convex function $f:K\to\R$ with $K\subset\R^{n-1}$ convex.
    \end{itemize}
    Now we can apply Fact~\ref{local smoothing}: there is a convex function $g:K\to\R$ above $f$ which is equal to $f$ in the neighborhood of $\partial K$ and Hessian convex $U\cup V$ where $V$ is the set of points where $f$ is already Hessian convex.
    
    Let $D'\subset D$ be the image under the affine chart of $\{(x,t):x\in K, g(x)\leq t<0\}$.
    Then we set
    $$
    \mc C'=\left(\mc C\smallsetminus\bigcup_\gamma \gamma D\right)\cup\bigcup_\gamma \gamma D'.
    $$
    And the $\Gamma$-equivariant map $\phi:\partial\mc C\to\partial\mc C'$ is defined by $\phi(x)=x$ if $x\in\partial\mc C\smallsetminus\bigcup_\gamma \gamma D$, and $\phi(\gamma\psi(x,f(x)))=\gamma\psi(x,g(x))$ for all $x\in K$ and $\gamma\in\Gamma$.
    One easily checks that it satisfies the conclusion of the corollary.
\end{proof}

Now we can prove the lemma.
\begin{proof}[Proof of Lemma~\ref{lem smoothing}]
 Since $\partial\mc C$ is differentiable outside $\Lambda$, and any face of $\partial\mc C$ intersecting $\Lambda$ must be contained in it, we have a partition of $\partial\mc C\smallsetminus\Lambda$ by its maximal closed faces, which have the form $\partial\mc C\cap H$ where $H$ is the tangent space to $\partial \mc C$ at some point of $\partial\mc C\smallsetminus\Lambda$.
 
 For any such maximal closed face $F$, the stabiliser $G_F\subset\Gamma$ is finite and preserves the tangent space $H$ at $F$, and $F\cap\gamma F=\emptyset$ for any $\gamma\in\Gamma\smallsetminus G_F$.
 Using, as in the previous proof, Fact~\ref{fact fixed point}, $G_F$ fixes a point $x$ in the relative interior of $F$ and preserves a codimension 1 subspace $L$ of $H$ that does not intersect $F$.
 It also fixes a point $y\in \R\PP^n\smallsetminus H$, as well as any point of the line through $x$ and $y$.
 
 Now any cap of $\overline{\mc C}$ containing $F$ and separated by a hyperplane spanned by $L$ and a point of $\Span(x,y)\cap \mr{int}\mc C$ is $G_F$-invariant.
 Moreover such a cap can be taken arbitrarily close to $F$, so we can find a cap $D_F$ such that $D_F\subset\mc C$, $D_F\cap A=\emptyset$, $D_F\cap\gamma D_F=\emptyset$ for any $\gamma\in\Gamma\smallsetminus G_F$, and $\gamma D_F=D_{\gamma F}$ for any $\gamma\in\Gamma$.
 Fix a compact neighborhood $U_F$ of $F$ in $D_F\cap\partial\mc C$.
 
 The interiors of $U_F$ where $F$ runs among all maximal closed faces covers $\partial\mc C\smallsetminus\Lambda$.
 We get that the interiors of $U_F/G_F$ embedded into $\partial\mc C\smallsetminus\Lambda/\Gamma$ form a covering, and we can select a countable subcover $U_{F_1}/G_{F_1},U_{F_2}/G_{F_2},\dots$ which is locally finite.
 Now we apply inductively Corollary~\ref{cor local smoothing}.
 
 First we apply it to $\mc C$ and $D_{F_1}$ to get a $\Gamma$-invariant $\mc C_1\subset \mc C$ containing $A$ and an equivariant homeomorphism $\phi_1:\partial\mc C\to\partial\mc C_1$ that is identity outside $\Gamma\cdot D_{F_1}$ and $\partial \mc C_1$ is Hessian convex at every point of $\phi_1(\Gamma\cdot U_{F_1})$.

 Then we apply it to $\mc C_1$ and $D_{F_2}\cap \mc C_1$ to get a $\Gamma$-invariant $\mc C_2\subset\mc C_1$ containing $A$ and an equivariant homeomorphism $\phi_2:\partial\mc C_1\to\partial\mc C_2$ such that, denoting $\psi_2=\phi_2\circ\phi_1$, we have that $\partial\mc C_2$ is Hessian convex at every point of $\psi_2(\Gamma(U_{F_1}\cup U_{F_2}))$.
 
 We continue the inductive procedure: having constructed $\mc C_k$ and $\psi_k:\partial\mc C\to\partial\mc C_k$, we apply Corollary~\ref{cor local smoothing} to $\mc C_k$ and $D_{F_{k+1}}\cap\mc C_k$ to get $A\subset\mc  C_{k+1}\subset\mc C_k$ and $\phi_{k+1}:\partial\mc C_k\to\partial\mc C_{k+1}$, and then we set $\psi_{k+1}=\phi_{k+1}\circ\psi_k$, so that $\partial\mc C_{k+1}$ is Hessian convex on $\psi_{k+1}(\Gamma(U_{F_1}\cup\dots\cup U_{F_{k+1}}))$.
 
 Note that $\Lambda\subset\partial\mc C_{k+1}$ and $\psi_{k+1}=\psi_k$ outside $\Gamma D_{F_{k+1}}$, including $\Lambda$.
 Since $(U_{F_k}/G_{F_k})_k$ is locally finite in $\partial\mc C\smallsetminus\Lambda/\Gamma$, for every $k$ there is $i_k$ such that $\psi_i=\psi_{i_k}$ on $\Gamma(D_{F_1}\cup\dots\cup D_{F_k})$ for all $i\geq i_k$.
 
 Let $\mc C'=\bigcap_k\mc C_k$ which is $\Gamma$-invariant.
 We have an equivariant homeomorphism $\psi:\partial\mc C\to\partial\mc C'$ which is equal to identity on $\Lambda$ and equal to $\psi_{i_k}$ on $\Gamma\cdot D_{F_k}$ for every $k$.
 As $\partial\mc C'$ coincides with $\partial\mc C_{i_k}$ on $\psi(\Gamma\cdot U_{F_k})$, it is Hessian convex there.
 Thus $\partial\mc C'$ is Hessian convex on $\partial\mc C'\smallsetminus\Lambda$.
\end{proof}

\bibliographystyle{alpha}
{\small \bibliography{bib}}

\end{document}